\documentclass[11pt]{amsart}

\usepackage{amscd,amssymb,amsopn,amsmath,amsthm,mathrsfs,graphics,amsfonts,enumerate,verbatim,calc
}

\usepackage[all]{xy}

\usepackage{color}

\usepackage[OT2,OT1]{fontenc}
\newcommand\cyr{%
\renewcommand\rmdefault{wncyr}%
\renewcommand\sfdefault{wncyss}%
\renewcommand\encodingdefault{OT2}%
\normalfont
\selectfont}
\DeclareTextFontCommand{\textcyr}{\cyr}

\usepackage{amssymb,amsmath}

\DeclareFontFamily{OT1}{rsfs}{}
\DeclareFontShape{OT1}{rsfs}{n}{it}{<-> rsfs10}{}
\DeclareMathAlphabet{\mathscr}{OT1}{rsfs}{n}{it}

\numberwithin{equation}{section}
\newtheorem{theorem}{Theorem}[section]
\newtheorem{lemma}[theorem]{Lemma}
\newtheorem{proposition}[theorem]{Proposition}
\newtheorem{corollary}[theorem]{Corollary}

\newtheorem{question}{Question}
\newtheorem{Problem}{Problem}

\newtheorem{maintheorema}{Main Theorem A}

\newtheorem{maintheoremb}{Main Theorem B}

\newtheorem{maintheoremc}{Main Theorem C}

\newtheorem{maintheoremd}{Main Theorem D}

\newtheorem{maincorollary1}{Corollary}

\theoremstyle{definition}
\newtheorem{definition}[theorem]{Definition}
\newtheorem{remark}[theorem]{Remark}
\theoremstyle{remark}

\newtheorem{example}[theorem]{Example}
\newtheorem{acknowledgement}{Acknowledgement}

\newcommand{\Ass}{\operatorname{Ass}}
\newcommand{\im}{\operatorname{Im}}
\renewcommand{\ker}{\operatorname{Ker}}

\newcommand{\Spec}{\operatorname{Spec}}

\newcommand{\Ht}{\operatorname{ht}}

\newcommand{\Ext}{\operatorname{Ext}}

\newcommand{\Supp}{\operatorname{Supp}}

\newcommand{\depth}{\operatorname{depth}}

\newcommand{\coker}{\operatorname{Coker}}

\newcommand{\fm}{\frak{m}}
\newcommand{\fp}{\frak{p}}
\newcommand{\fq}{\frak{q}}
\newcommand{\fa}{\frak{a}}

\begin{document}
\title[$F$-injectivity and Frobenius closure of ideals in Noetherian rings]
{$F$-injectivity and Frobenius closure of ideals in Noetherian rings of characteristic $p>0$}

\author[P. H. Quy]{Pham Hung Quy}
\address{Department of Mathematics, FPT University, Hoa Lac Hi-Tech Park, Ha Noi, Viet Nam}
\email{quyph@fe.edu.vn}

\author[K. Shimomoto]{Kazuma Shimomoto}
\address{Department of Mathematics, College of Humanities and Sciences, Nihon University, Setagaya-ku, Tokyo 156-8550, Japan}
\email{shimomotokazuma@gmail.com}
\thanks{
P.H. Quy is partially supported by a fund of Vietnam National Foundation for Science
and Technology Development (NAFOSTED) under grant number
101.04-2017.10. This paper was written while the first author was visiting Vietnam Institute for Advanced Study in Mathematics. He would like to thank the VIASM for hospitality and financial support. \\
K. Shimomoto is partially supported by Grant-in-Aid for Young Scientists (B) \# 25800028.}



\begin{abstract}
The main aim of this article is to study the relation between $F$-injective singularity and the Frobenius closure of parameter ideals in Noetherian rings of positive characteristic. The paper consists of the following themes, including many other topics.
\begin{enumerate}
\item
We prove that if every parameter ideal of a Noetherian local ring of prime characteristic $p>0$ is Frobenius closed, then it is $F$-injective.

\item
We prove a necessary and sufficient condition for the injectivity of the Frobenius action on $H^i_{\fm}(R)$ for all $i \le f_{\fm}(R)$, where $f_{\fm}(R)$ is the finiteness dimension of $R$. As applications, we prove the following results. (a) If the ring is $F$-injective, then every ideal generated by a filter regular sequence, whose length is equal to the finiteness dimension of the ring, is Frobenius closed. It is a generalization of a recent result of Ma and which is stated for generalized Cohen-Macaulay local rings. (b) Let $(R,\fm,k)$ be a generalized Cohen-Macaulay ring of characteristic $p>0$. If the Frobenius action is injective on the local cohomology $H_{\fm}^i(R)$ for all $i < \dim R$, then $R$ is Buchsbaum. This gives an answer to a question of Takagi.

\item
We consider the problem when the union of two $F$-injective closed subschemes of a Noetherian $\mathbb{F}_p$-scheme is $F$-injective. Using this idea, we construct an $F$-injective local ring $R$ such that $R$ has a parameter ideal that is not Frobenius closed. This result adds a new member to the family of $F$-singularities.

\item
We give the first ideal-theoretic characterization of $F$-injectivity in terms the Frobenius closure and the limit closure. We also give an answer to the question about when the Frobenius action on the top local cohomology is injective.
\end{enumerate}
\end{abstract}

\maketitle

\begin{center}
{\textit{Dedicated to Prof. Shiro Goto on the occasion of his 70th birthday}}
\end{center}

\tableofcontents

\section{Introduction}

In this paper, we study the behavior of the Frobenius closure of ideals generated by a system of parameters (which we call a \textit{parameter ideal}) of a given Noetherian ring containing a field of characteristic $p>0$. Then we investigate how the \textit{F-injectivity} condition is related to the Frobenius closure of parameter ideals. Recall that a local ring $(R, \fm)$ of positive characteristic is $F$-injective if the natural Frobenius action on all local cohomology modules $H^i_{\fm}(R)$ are injective (cf. \cite{F83}). $F$-injective rings together with $F$-regular, $F$-rational and $F$-pure rings are the main objects of the family of singularities defined by the Frobenius map and they are called the \textit{F-singularities}. $F$-singularities appear in the theory of \textit{tight closure} (cf. \cite{H96} for its introduction), which was systematically introduced by Hochster and Huneke around the mid 80's \cite{HH90} and developed by many researchers, including Hara, Schwede, Smith, Takagi, Watanabe, Yoshida and others. A recent active research of $F$-singularities is centered around the correspondence with the singularities of the minimal model program. We recommend \cite{TW14} as an excellent survey for recent developments. It should be noted that the class of $F$-injective rings is considered to be the largest among other notable classes of $F$-singularities.

Under mild conditions of rings, $F$-regularity, $F$-rationality and $F$-purity can be checked by computing either the tight closure, or the Frobenius closure of (parameter) ideals. However, there was no known characterization of $F$-injectivity in terms of a closure operation of (parameter) ideals. If $R$ is Cohen-Macaulay, Fedder proved that $R$ is $F$-injective if and only if every parameter ideal is Frobenius closed. More than thirty years later after Fedder's work appeared, Ma extended Fedder's result for the class of \textit{generalized Cohen-Macaulay local rings}. Therefore, it is quite natural to ask the following question (cf. \cite[Remark 3.6]{M15}):

\begin{question}
\label{questionMa}
Is it true that a local ring is $F$-injective if and only if every parameter ideal is Frobenius closed?
\end{question}

Both authors of the present paper are very interested in the works of Ma in \cite{M14} and \cite{M15}. We started our join work when the first author was able to prove one direction of the above question. Indeed, this is the first main result of this paper and stated as follows.

\begin{maintheorema}[Theorem \ref{T1.2}]
Let $(R,\fm,k)$ be a local ring of characteristic $p>0$. Assume that every parameter ideal is Frobenius closed. Then $R$ is $F$-injective.
\end{maintheorema}

However, we prove that the converse of this theorem does not hold by constructing an explicit example (see Theorem \ref{counterexample}). It should be noted that the example in Theorem \ref{counterexample} comes from our geometrical consideration of patching two $F$-injective closed subschemes. More precisely, we prove the following theorem.

\begin{maintheoremb}[Example \ref{E6.1}, Theorem \ref{counterexample} and Corollary \ref{F-example}]
There exists an $F$-finite local ring $(R,\fm,k)$ of characteristic $p>0$ which is non-Cohen-Macaulay, $F$-injective, but not $F$-pure. Moreover, $R$ has a parameter ideal that is not Frobenius closed.
\end{maintheoremb}

Thus, the above theorems give an answer to Question \ref{questionMa} in a complete form, and it seems reasonable to call $(R,\fm,k)$ \textit{parameter F-closed} if every parameter ideal of $R$ is Frobenius closed (cf. Definition \ref{parameterF-closed}). It is true that every $F$-pure local ring is parameter $F$-closed and every parameter $F$-closed ring is $F$-injective as a consequence of our main theorems. However, our example in Main Theorem B may not be optimal, since the ring is not equidimensional. There is some hope that Question \ref{questionMa} has an affirmative answer when the ring is equidimensional with additional mild conditions. We note that a generalized Cohen-Macaulay ring is equidimensional. Inspired by Ma's work on generalized Cohen-Macaulay rings, we study $F$-injective rings in connection with the finiteness dimension. Recall that the finiteness dimension of $R$ is defined as
$$
f_{\fm}(R):=\inf \{i \,|\, H^i_{\fm}(R)~\text{is not finitely generated} \} \in \mathbb{Z}_{\ge 0} \cup \{\infty\}.
$$
We have the following theorem.

\begin{maintheoremc}[Theorem \ref{finitedim}]
Let $(R, \fm)$ be a reduced $F$-finite local ring of characteristic $p>0$ with $f_{\fm}(R)=t$ and let $s \le t$ be a positive integer. Then the following statements are equivalent.
\begin{enumerate}
\item
The Frobenius action on $H^i_{\fm}(R)$ is injective for all $i \le s$.

\item
Every filter regular sequence $x_1, \ldots, x_s$ of both $R$ and $R^{1/p}/R$ generates a Frobenius closed ideal of $R$ and it is a standard sequence on $R$.

\item
Every filter regular sequence $x_1, \ldots, x_s$ of both $R$ and $R^{1/p}/R$ generates a Frobenius closed ideal of $R$.
\end{enumerate}
\end{maintheoremc}

We deduce many interesting results from Main Theorem C. First, we generalize Ma's results \cite{M15} in terms of finiteness dimension.

\begin{maincorollary1}[Theorem \ref{maintheoremA} and Corollary \ref {Buchsbaum}]
Let $(R,\fm,k)$ be an $F$-injective local ring with $f_{\fm}(R)=t$. Then every filter regular sequence of length $t$ is a standard sequence, and every ideal generated by a filter regular sequence of length at most $t$ is Frobenius closed. If $R$ is $F$-injective and generalized Cohen-Macaulay, then every parameter ideal is Frobenius closed and $R$ is Buchsbaum.
\end{maincorollary1}

The last assertion in the above corollary is Ma's affirmative answer to a question of Takagi. Ma's proof follows from his result on the equivalence between the class of $F$-injective rings and the class of parameter $F$-closed rings for a generalized Cohen-Macaulay ring $R$, together with the result of Goto and Ogawa in \cite{GO83}, who showed that if a parameter $F$-closed local ring (in our terminology) is generalized Cohen-Macaulay, then it is Buchsbaum. Using Main Theorem C and Goto-Ogawa's argument, we prove the following corollary.

\begin{maincorollary1}[Corollary \ref{takagi}]
Let $(R, \fm)$ be a reduced $F$-finite generalized Cohen-Macaulay local ring with $d=\dim R$. Suppose that the Frobenius action on $H^i_{\fm}(R)$ is injective for all $i<d$. Then $R$ is Buchsbaum.
\end{maincorollary1}

Finally, we state a result that gives an ideal-theoretic characterization of $F$-injectivity via the notion of {\it limit closure}. The notion of limit closure appears naturally when we consider local cohomology as the direct limit of Koszul cohomology for non-Cohen-Macaulay rings. The limit closure of a sequence of elements $x_1, \ldots, x_t$ in a ring $R$ is defined as follows
$$
(x_1,\ldots,x_t)^{\lim}=\bigcup_{n>0}\big((x_1^{n+1},\ldots,x_t^{n+1}):_R (x_1\cdots x_t)^n \big)
$$
with the convention that $(x_1,\ldots,x_t)^{\lim} = 0$ when $t=0$. Note that $(x_1,\ldots,x_t)^{\lim}$ is an ideal of $R$. We prove the following theorem.

\begin{maintheoremd}[Theorem \ref{toplc} and Theorem \ref{characterization}]
Let $(R, \fm)$ be a local ring of characteristic $p>0$ and of dimension $d>0$. Then we have the following statements.

\begin{enumerate}
\item
The Frobenius action on the top local cohomology $H^d_{\fm}(R)$ is injective if and only if $\fq^F \subseteq \fq^{\lim}$ for all parameter ideals $\fq$.

\item The following statements are equivalent.
\begin{enumerate}
\item
$R$ is $F$-injective

\item
For every filter regular sequence $x_1,\ldots,x_d$, we have
$$
(x_1,\ldots,x_t)^F \subseteq (x_1,\ldots,x_t)^{\lim}
$$
for all $0 \le t \le d$.
\item
There is a filter regular sequence $x_1,\ldots,x_d$ such that
$$
(x_1^n,\ldots,x_t^n)^F \subseteq (x_1^n,\ldots,x_t^n)^{\lim}
$$
for all $0 \le t \le d$ and for all $n \ge 1$.
\end{enumerate}
\end{enumerate}
\end{maintheoremd}

Main Theorem D is not only a generalization of Main Theorem A, but it also helps us better understand Main Theorems B and C. As other side topics, we also discuss problems such as non-$F$-injective locus and localizations of $F$-parameter closed rings. Many questions concerning $F$-injectivity are addressed in the last section. \\

The main technique of this paper is to analyze the local cohomology modules by filter regular sequence via the Nagel-Schenzel isomorphism (cf. Lemma \ref{nagel-shenzel}). It is worth noting that the notion of filter regular sequence has arisen from the theory of generalized Cohen-Macaulay ring in \cite{CST78} and has become a powerful tool in many problems of commutative algebra nowadays. The authors hope that the present paper will shed light on the connection between tight closure theory and non-Cohen-Macaulay rings via techniques developed in this article. The structure of this paper is as follows.

In \S~\ref{sec1}, we introduce notation and give a brief review on Frobenius closure of ideals, Frobenius action, and the local cohomology modules.

In \S~\ref{sec2}, We recall some standard results on the Frobenius closure and give complete proofs to them for the convenience of readers. An important fact to keep in mind is that the formation of Frobenius closure of ideals and $F$-injectivity commute with localization. After that, we recall the notion of filter regular sequence. The reader will find that this notion, together with the Nagel-Schenzel isomorphism, will play a prominent role in proving the injectivity-type results on local cohomology modules under the Frobenius action. The main theorem A will be proven in this section (cf. Theorem \ref{T1.2}).

In \S~\ref{sec3}, firstly, we review the definition of generalized Cohen-Macaulay rings and Buchsbaum rings. They form a wider class than that of Cohen-Macaulay rings. The notion of standard parameter ideal plays an important role in the theory of generalized Cohen-Macaulay rings. Then we combine these notions with the notion of the finiteness dimension. We prove Main Theorem C in this section (cf. Theorem \ref{finitedim}). Among many consequences, we recover Ma's result on generalized Cohen-Macaulay $F$-injective rings (cf. Theorem \ref{maintheoremA} and Corollary \ref{takagi}).

In \S~\ref{sec4}, we compare $F$-injective and $F$-pure rings. In geometrical setting, we consider the problem when the union of two $F$-injective closed subschemes is again $F$-injective (cf. Theorem \ref{geometry}). Our result is useful in the construction of certain local rings in characteristic $p>0$.

In \S~\ref{sec5}, we construct some interesting examples. The main result in this section is that there exists a local ring that is $F$-injective with a parameter ideal that is not Frobenius closed (cf. Theorem \ref{counterexample}). Then we define the parameter $F$-closed rings as a new member of $F$-singularities.

In \S~\ref{sec6}, we prove an ideal-theoretic characterization of $F$-injectivity using the limit closure (cf. Theorem \ref{characterization}). Then we consider the injectivity of the Frobenius action on the top local cohomology (cf. Theorem \ref{toplc}). The readers are encouraged to ponder on this characterization to shed light on the previous results (some results of \cite{CQ14} may be helpful).

In \S~\ref{sec7}, we make a list of some open problems for the future research.

\begin{acknowledgement}
We are deeply grateful to Linquan Ma for his inspiring works. Indeed, Corollary \ref{takagi} was born from the useful and fruitful discussions with him. We are also grateful to Prof. Shihoko Ishii for a valuable comment. The authors are grateful to the referee for his/her carefully reading and useful comments.
\end{acknowledgement}

\section{Notation and conventions}
\label{sec1}
In this paper, all rings are (Noetherian) commutative with unity. A \textit{local ring} is a commutative Noetherian ring with the unique maximal ideal $\fm$. Denote a local ring by $(R,\fm,k)$. Let $M$ be a finitely generated module over a local ring $(R,\fm,k)$. We say that an ideal $\fq \subseteq R$ is a \textit{parameter ideal} of $M$, if $\fq$ is generated by a system of parameters of $M$. We say that a ring $R$ is \textit{equidimensional}, if $\dim R=\dim R/\fp$ for all minimal primes $\fp$ of $R$. Let $\underline{x}:=x_1,\ldots,x_n$ be a sequence of elements in a ring $R$. For an $R$-module $M$, let $H^i(\underline{x};M)$ denote the \textit{i-th Koszul cohomology module} of $M$ with respect to $\underline{x}$. We employ the convention that $\dim M=-1$ if $M$ is a trivial module.

Let $R$ be a Noetherian ring containing a field of characteristic $p>0$. Let $F:R \to R$ denote the Frobenius endomorphism (the $p$-th power map). We say that a Noetherian ring $R$ is \textit{$F$-finite} if $F:R \to R$ is module-finite. Let us recall the definition of the Frobenius closure with its properties briefly. Let $I=(x_1,\ldots,x_t)$ be an ideal of $R$. The \textit{Frobenius closure} of $I$, denoted by $I^F$, is defined to be the set of all elements of $R$ satisfying the following property: $u \in I^F$ if and only if $u^q \in I^{[q]}$ for $q=p^e \gg 0$, where $I^{[q]}:=(x_1^q,\ldots,x_t^q)$. Indeed, $I^F$ is an ideal of $R$. If $I$ is a parameter ideal of a local ring $R$, then so is $I^{[q]}$. Let $R$ be a Noetherian ring with $I$ its proper ideal. Then we denote by $H^i_I(R)$ the \textit{i-th local cohomology module} with support at $I$ (cf. \cite{BS98} and \cite{ILL07} for local cohomology modules). Recall that local cohomology may be computed as the homology of the \v{C}ech complex
$$
0 \to R \to \bigoplus_{i=1}^t R_{x_i} \to \cdots \to R_{x_1 \ldots x_t} \to 0.
$$
Let $R$ be a Noetherian ring of characteristic $p>0$ with an ideal $I=(x_1,\ldots,x_t)$. Then the Frobenius endomorphism $F:R \to R$ induces a natural Frobenius action $F_*:H^i_I(R) \to H^i_{I^{[p]}}(R) \cong H^i_{I}(R)$ (cf. \cite{BH98} for this map). There is a very useful way of describing the top local cohomology. It can be given as the direct limit of Koszul cohomologies
$$
H^t_I(R) \cong \lim_{\longrightarrow} R/(x_1^n, \ldots, x_t^n).
$$
Then for each $\overline{a} \in H^t_{I}(R)$, which is the canonical image of $a+(x_1^n, \ldots, x_t^n)$, we find that $F_*(\overline{a})$ is the canonical image of $a^p +(x_1^{pn}, \ldots, x_t^{pn})$.

By a \textit{Frobenius action} on local cohomology modules, we always mean the one as defined above. A local ring $(R,\fm,k)$ is \textit{$F$-injective} if the Frobenius action on $H^i_{\fm}(R)$ is injective for all $i \ge 0$. Assume that $R$ is a reduced ring of characteristic $p>0$ with minimal prime ideals $\fp_1,\ldots, \fp_r$. Consider the natural inclusions:
$$
R \hookrightarrow \prod_{i=1}^r R/\fp_i \hookrightarrow \prod_{i=1}^r \overline{Q(R/\fp_i)},
$$
where $\overline{Q(R/\fp_i)}$ is the algebraic closure of the quotient field $Q(R/\fp_i)$ of $R/\fp_i$. We define
$$
R^{1/p^e}:= \big \{ x \in \prod_{i=1}^r \overline{Q(R/\fp_i)} \,\, \big | \,\, x^{p^e} \in R \big \}.
$$
We note that $(R,\fm)$ is $F$-injective if and only if the inclusion $R \hookrightarrow R^{1/p^e}$ induces the injective map of local cohomology $H^i_{\fm}(R) \hookrightarrow H^i_{\fm}(R^{1/p^e})$ for all $i \ge 0$. We will review definitions and results from generalized Cohen-Macaulay rings in \S~\ref{sec3}.

\section{Frobenius action on local cohomology modules}
\label{sec2}

\subsection{Frobenius closure of ideals}
We collect basic known facts on the Frobenius closure of ideals in a Noetherian ring of characteristic $p>0$ for the convenience of readers.

\begin{lemma}\label{reduced}
Let $(R,\fm,k)$ be a local ring of characteristic $p>0$. If there is an element $x \in R$ such that $(x^n)$ is Frobenius closed for all $n >0$, then $R$ is reduced.
\end{lemma}

\begin{proof}
Assume that we have $u^m=0$ for $m>0$ and $u \in R$. Then we have $u^{q} \in (x^n)^{q}$ for all $q=p^e \ge m$ and $u \in (x^n)^F=(x^n)$. Hence $u \in \bigcap_{n>0} (x^n)=(0)$ by Krull's intersection theorem.
\end{proof}

\begin{lemma}
\label{L1.6}
Let $x_1,\ldots,x_k$ be a sequence of elements in a Noetherian ring $R$ of characteristic $p>0$ such that $x_1,\ldots,x_k$ is a regular sequence in any order and $(x_1,\ldots, x_k)$ is Frobenius closed. Then $(x_1^{n_1},\ldots,x_k^{n_k})$ is Frobenius closed for all integers $n_1,\ldots,n_k \ge 1$.
\end{lemma}

\begin{proof}
It is enough to prove that $(x_1^{n_1}, x_2,\ldots,x_k)$ is Frobenius closed for all $n_1 \ge 1$. We proceed by induction on $n_1$. The case $n_1 = 1$ is trivial. For $n_1 >1 $, let us take $a \in (x_1^{n_1}, x_2,\ldots,x_k)^F$. Then $a^q \in (x_1^{n_1}, x_2,\ldots,x_k)^{[q]} \subseteq (x_1, x_2,\ldots,x_k)^{[q]}$ for $q=p^e \gg 0$. Therefore $a \in (x_1, x_2,\ldots,x_k)^F = (x_1, x_2,\ldots,x_k)$. So $a = b_1x_1+\cdots+b_k x_k$. We have
$$
a^q = b_1^qx_1^q+\cdots+b_k^qx_k^q \in (x_1^{n_1}, x_2,\ldots,x_k)^{[q]}
$$
and $b_1^qx_1^q \in (x_1^{n_1}, x_2,\ldots,x_k)^{[q]}$. Hence $b_1^qx_1^q-cx_1^{n_1q} \in (x_2,\ldots,x_k)^{[q]}$ for some $c$. Since $x_1,\ldots,x_k$ is a regular sequence in any order, we have $b_1^q - cx_1^{(n_1-1)q} \in (x_2,\ldots,x_k)^{[q]}$ and $b_1^q \in (x_1^{n_1-1},x_2,\ldots,x_k)^{[q]}$. Therefore, $b_1 \in (x_1^{n_1-1},x_2,\ldots,x_k)^F=(x_1^{n_1-1},x_2,\ldots,x_k)$ by induction hypothesis on $n_1$. Hence $a=b_1x_1+\cdots+b_kx_k \in (x_1^{n_1},x_2,\ldots,x_k)$, as required.
\end{proof}

\begin{lemma}
\label{commute}
Frobenius closure commutes with localization. In particular, localization of a Frobenius closed ideal is Frobenius closed.
\end{lemma}

\begin{proof}
Let $J \subseteq R$ be an ideal. Then $u^q \in J^{[q]}$ for $q=p^e \gg 0$ if and only if $u \in JR^{\infty} \cap R$, where $R^{\infty}$ is the perfect closure of $R$; $R^{\infty}$ is the direct limit of $\{R \to R \to R \to \cdots\}$, where $R \to R$ is the Frobenius map. Let $\phi:R \to R^{\infty}$ be the natural ring map and write $JR^{\infty} \cap R$ for $\phi^{-1}(JR^{\infty}) \cap R$ for simplicity. Hence $J^F=JR^{\infty} \cap R$. Let $S \subseteq R$ be a multiplicative set. Since the localization functor is exact, we have
$$
S^{-1}R^{\infty} \cong (S^{-1}R)^{\infty}.
$$
Then we have
$$
S^{-1}(J^F)=S^{-1}(JR^{\infty} \cap R)=J(S^{-1}R^{\infty}) \cap S^{-1}R=(S^{-1}J)^F,
$$
as claimed.
\end{proof}

\subsection{Filter regular sequence}
Let us recall the definition of filter regular sequence. We always assume that a module is finitely generated over a ring.

\begin{definition}
Let $M$ be a finitely generated module over a local ring $(R,\fm,k)$ and let $x_1,\ldots,x_t$ be a set of elements of $R$. Then we say that $x_1,\ldots,x_t$ is a \textit{filter regular sequence} on $M$ if the following conditions hold:
\begin{enumerate}
\item
We have $(x_1,\ldots,x_t) \subseteq \fm$.

\item
We have $x_i \notin \frak p$ for all $\fp \in \Ass_R\Big(\dfrac{M}{(x_1,\ldots,x_{i-1})M}\Big) \setminus \{\fm\},~i=1,\ldots,t$.
\end{enumerate}
\end{definition}

We can deduce the existence of filter regular sequence, using the prime avoidance lemma. For this fact, we refer the reader to \cite[Remark 4.5]{Q13}. Note that the filter regular sequence that we just defined is also called an \textit{$\fm$-filter regular sequence} in other literatures.

\begin{lemma}
\label{filter}
Let $M$ be a finitely generated module over a local ring $(R,\fm,k)$. Then $x_1,\ldots, x_t \in \fm$ form a filter regular sequence on $M$ if and only if one of the following conditions holds:
\begin{enumerate}
\item
The quotient
$$
\dfrac{\big((x_1,\ldots,x_{i-1})M:_M x_i\big)}{(x_1,\ldots,x_{i-1})M}
$$
is an $R$-module of finite length for $i=1,\ldots,t$.

\item
Fix $i \in \mathbb{N}$ with $1 \le i \le t$. Then the sequence
$$
\frac{x_1}{1},\frac{x_2}{1},\ldots,\frac{x_i}{1}
$$
forms an $R_{\fp}$-regular sequence in $M_{\fp}$ for every $\fp \in \big(\Spec (R/(x_1,\ldots,x_i)) \cap \Supp_R M\big) \setminus \{\fm\}$.

\item
The sequence $x_1^{n_1},\ldots,x_t^{n_t}$ is a filter regular sequence for all $n_1, \ldots, n_t \ge 1$.
\end{enumerate}
\end{lemma}

\begin{proof}
The proof is found in \cite[Proposition 2.2]{NS95}.
\end{proof}

The following result is very useful in this paper (cf. \cite[Proposition 3.4]{NS95}).

\begin{lemma}[Nagel-Schenzel isomorphism]
\label{nagel-shenzel}
Let $(R,\fm,k)$ be a local ring and let $M$ be a finitely generated $R$-module. Let $x_1,\ldots,x_t$ be a filter regular sequence on $M$. Then we have
\[
H^i_{\fm}(M) \cong
\begin{cases}
H^i_{(x_1,\ldots, x_t)}(M) &\text{ if } i<t\\
H^{i-t}_{\fm}(H^t_{(x_1, \ldots, x_t)}(M)) &\text{ if } i \ge t.
\end{cases}
\]
\end{lemma}

The extremely useful case of Nagel-Schenzel's isomorphism is that we can consider $H^t_{\fm}(R)$ as the submodule $H^{0}_{\fm}(H^t_{(x_1, \ldots, x_t)}(M))$ of $H^t_{(x_1, \ldots, x_t)}(M)$. These module structures are compatible with the Frobenius actions (in positive characteristic). Moreover, $H^t_{(x_1, \ldots, x_t)}(M)$ is the top local cohomology module whose Frobenius action is described explicitly.

\subsection{Frobenius closed parameter ideals I}
Now we prove the following theorem.

\begin{theorem}
\label{T1.2}
Let $(R,\fm,k)$ be a local ring of characteristic $p>0$. Set $d=\dim R$. Then we have the following results:
\begin{enumerate}
\item
Assume that $x_1,\ldots,x_t$ is a filter regular sequence on $R$ such that $(x^{p^n}_1,\ldots,x^{p^n}_t)$ is Frobenius closed for all $n \ge 0$. Then the Frobenius action on $H^t_{\fm}(R)$ is injective.

\item
If every parameter ideal of $R$ is Frobenius closed, then $R$ is $F$-injective.
\end{enumerate}
\end{theorem}

\begin{proof}
(1): Set $I=(x_1,\ldots,x_t)$. Then $I^{[p^n]}$ is Frobenius closed by assumption. Then we have the following commutative diagram:
$$
\begin{CD}
R/I @>>>  R/I^{[p]}  @>>> R/I^{[p^2]} @>>> \cdots \\
@VFVV @VFVV @VFVV \\
R/I^{[p]} @>>> R/I^{[p^2]} @>>> R/I^{[p^3]} @>>> \cdots
\end{CD}
$$
where each vertical map is the Frobenius and each map in the horizontal direction is multiplication map by $(x_1 \cdots x_t)^{p^e-p^{e-1}}$ in the corresponding spot. The direct limits of both lines are $H^t_I(R)$ and the vertical map is exactly the Frobenius action on $H^t_I(R)$. Since $I^{[p^n]}$ is Frobenius closed, each vertical map is injective. Hence the direct limit map is injective. Therefore, the Frobenius acts injectively on $H^t_I(R)$. To show that Frobenius acts injectively on $H^t_{\fm}(R)$, we need Nagel-Schenzel isomorphism:
$$
H^t_{\fm}(R) \cong H^0_{\fm}\big(H^t_I(R)\big).
$$
Thus, the Frobenius action on $H^t_{\fm}(R)$ is the direct limit of the following direct system
$$
\begin{CD}
H^0_{\fm}(R/I) @>>> H^0_{\fm}(R/I^{[p]}) @>>> H^0_{\fm}(R/I^{[p^2]}) @>>> \cdots \\
@VFVV @VFVV @VFVV \\
H^0_{\fm}(R/I^{[p]}) @>>> H^0_{\fm}(R/I^{[p^2]}) @>>> H^0_{\fm}(R/I^{[p^3]}) @>>> \cdots \\
\end{CD}
$$
proving that $R$ is $F$-injective, as claimed.

(2): There exists a filter regular sequence $x_1,\ldots,x_d$ that is a system of parameters of $R$ by prime avoidance lemma. Let $I=(x_1,\ldots,x_t)$. Then by \cite[Lemma 3.1]{M15}, the ideals $I$ and $I^{[p^n]}$ are both Frobenius closed for $0 \le t \le d$ and $n \ge 0$. Using this together with the above discussions, we conclude that Frobenius acts injectively on $H^t_{\fm}(R)$.
\end{proof}

\begin{remark}
Let $(R,\fm,k)$ be a local ring of characteristic $p>0$ with a regular element $x \in \fm$. Assume that every parameter ideal of $R/xR$ is Frobenius closed. Then we claim that the Frobenius action on $H^t_{\fm}(R)$ is injective, where $t=\depth R$. Indeed, $R/xR$ is $F$-injective by Theorem \ref{T1.2} and the claim follows by \cite[Lemma A1]{HMS14}. We will slightly generalize this statement as Corollary \ref{Frobeniuslocal}.
\end{remark}

\begin{corollary}\label{C1.9}
Let $(R,\fm,k)$ be a Cohen-Macaulay local ring of characteristic $p>0$. Then the following are equivalent:
\begin{enumerate}
\item
Every parameter ideal of $R$ is Frobenius closed.

\item
There is a parameter ideal of $R$ that is Frobenius closed.

\item
$R$ is $F$-injective.
\end{enumerate}
\end{corollary}

\begin{proof}
(1) $\Rightarrow$ (2) is trivial, (2) $\Rightarrow$ (3) by Theorem \ref{T1.2} and Lemma \ref{L1.6} and (3) $\Rightarrow$ (1) by \cite[Lemma 10.3.20]{BH98}.
\end{proof}

\begin{remark}
\label{dualizing}
Let $(R,\fm,k)$ be a reduced $F$-finite local ring. Then by a theorem of Kunz \cite{Ku76}, $R$ is an excellent ring. Hence the $\fm$-adic completion $\widehat{R}$ is also reduced and $F$-finite. Since $R$ is $F$-finite, it is known that it is a homomorphic image of a regular local ring by Gabber \cite[Remark 13.6]{G05}. Hence $R$ admits a dualizing complex.
\end{remark}

In order to reduce a certain problem on a complete local ring of characteristic $p>0$ with an arbitrary residue field to the problem on an $F$-finite local ring, the ``$\Gamma$-construction'', due to Hochster and Huneke \cite{HH94}, will be useful. We briefly recall the construction. Let $(R,\fm,k)$ be a complete local ring with coefficient field $k$ of characteristic $p>0$ and of dimension $d$. Let us fix a $p$-basis $\Lambda$ of $k \subset R$ and let $\Gamma \subset \Lambda$ be a cofinite subset (we refer the reader to \cite{Ma86} for the definition of a $p$-basis). We denote by $k_e$ (or $k_{\Gamma,e}$ to signify the dependence on the choice of $\Gamma$) the purely inseparable  extension field $k[\Gamma^{1/p^e}]$ of $k$, which is obtained by adjoining $p^e$-th roots of all elements in $\Gamma$. Next, fix a system of parameters $x_1,\ldots,x_d$ of $R$. Then the natural map $A:=k[[x_1,\ldots,x_d]] \to R$ is module-finite. Let us put
$$
A^{\Gamma}:=\bigcup_{e>0} k_e[[x_1,\ldots,x_d]].
$$
Then the natural map $A \to A^{\Gamma}$ is faithfully flat and purely inseparable and $\fm_A A^{\Gamma}$ is the unique maximal ideal of $A^{\Gamma}$. Now we set $R^{\Gamma}:=A^{\Gamma} \otimes_A R$. Then $R \to R^{\Gamma}$ is faithfully flat and purely inseparable. The crucial fact about $R^{\Gamma}$ is that it is an $F$-finite local ring (see \cite[(6.6) Lemma]{HH94}).

\begin{lemma}
\label{L1.11}
Let $(R,\fm,k)$ be an $F$-injective local ring. Then $R$ is a reduced ring. If furthermore $R$ is $F$-finite, then $R_{\fp}$ is $F$-injective for all $\fp \in \Spec R$.
\end{lemma}

\begin{proof}
We have a faithfully flat extension:
$$
R \to \widehat{R}^{\Gamma} \to S:=\widehat{\widehat{R}^{\Gamma}},
$$
where $\widehat{R}$ is the completion of $R$ and $\widehat{R}^{\Gamma}$ is its $\Gamma$-construction. Moreover,
$\widehat{R}^{\Gamma}$ is $F$-finite and $F$-injective for all sufficiently small choice of $\Gamma$ by \cite[Lemma 2.9]{EH08}. If we can show that $S$ is reduced, then $R$ is also reduced. So we may henceforth assume that $R$ is $F$-finite and $F$-injective. Note that $R$ has a dualizing complex by Remark \ref{dualizing}. Then the fact that $R$ is reduced is found in \cite[Remark 2.6]{SchZ13} and the fact that $R_{\fp}$ is $F$-injective is found in \cite[Proposition 4.3]{Sch09}.
\end{proof}

The following is of independent interest.

\begin{proposition}
Let $(R,\fm,k)$ be a reduced $F$-finite local ring and let
$$
U:=\{\fp \in \Spec R~|~R_{\fp}~\mbox{is}~F\mbox{-injective}\}.
$$
Then $U$ is a Zariski open subset of $\Spec R$.
\end{proposition}

\begin{proof}
We may present $R$ as a homomorphic image of a regular local ring $S$ of Krull dimension $n$. Set $d=\dim R$. Then $R_{\fp}$ is $F$-injective if and only if the natural map
$$
H^i_{\fp R_{\frak p}}(R_{\fp}) \to H^i_{\fp R_{\fp}}(R^{1/p}_{\fp})
$$
is injective for all $i \le \dim R_{\fp}$. Let $P$ be the preimage of $\fp$ in $S$ under the surjection $S \twoheadrightarrow R$. Note that $\dim S_P=n-\dim R/\fp$ ($S$ is a catenary domain). By Grothendieck's local duality theorem, the map
$$
\Ext^{n-\dim R/\fp - i}_{S_P}(R^{1/p}_{\fp}, S_P) \to \Ext^{n-\dim R/\fp - i}_{S_P}(R_{\fp}, S_P)
$$
is surjective for all $i \le \dim R_{\fp}$. For each $i \le d$, we set
$$
C_i:=\coker \big(\Ext^{n - i}_{S}(R^{1/p}, S) \to \Ext^{n - i}_{S}(R, S)\big).
$$
Therefore, $R_{\fp}$ is $F$-injective if and only if $\fp \notin \bigcup_{i = 0}^d \Supp(C_i)$, which is a closed subset in $\Spec R$ since $C_i$ is finitely generated for all $i \le d$.
\end{proof}

\begin{proposition}
\label{P1.13}
Assume that $(R,\fm,k)$ is an $F$-finite $F$-injective local ring. Then every ideal generated by a regular sequence is Frobenius closed.
\end{proposition}

\begin{proof}
By Remark \ref{dualizing}, $R$ has a dualizing complex. Let $t=\depth R$ and $d=\dim R$. Since the length of every regular sequence of maximal length is equal to $t$, it is enough to prove that every ideal generated by a regular sequence $x_1,\ldots,x_t$ is Frobenius closed by \cite[Lemma 3.1]{M15}. We proceed by induction on $d$. The case $d = 1$ follows from Corollary \ref{C1.9}, since $R$ is Cohen-Macaulay. For $d>1$, if $t = d$, then we use Corollary \ref{C1.9} again. Therefore, we can assume henceforth that $t<d$. Set $I = (x_1, \ldots, x_t)$. We have
$$
H^t_I(R) \cong \varinjlim R/I^{[q]}.
$$
Since $I$ is generated by a regular sequence, all the maps in the direct system are injective. Thus, the natural map $R/I \to H^t_I(R)$ is injective. Suppose that $a \in R$ satisfies $a^q \in I^{[q]}$ for some $q = p^e$. Let $\overline{a}$ be the image of $a +I \in R/I$ in $H^t_I(R)$. Then $\overline{a}$ is nilpotent under the Frobenius action on $H^t_I(R)$. Let $N \subseteq R/I$ be the cyclic $R$-module generated by $a+I \in R/I$. Let $\fp$ be a prime ideal such that $I \not\subseteq \fp$ and $\fp \neq \fm$. Then $N_{\fp}=0$ quite evidently. Let $\fp$ be a prime ideal such that $I \subseteq \fp$ and $\fp \neq \fm$. Then $R_\fp$ is $F$-injective by Lemma \ref{L1.11} and $x_1,\ldots,x_t$ is also an $R_{\fp}$-regular sequence. By induction hypothesis, $IR_{\fp}$ is Frobenius closed and we have $a \in IR_{\fp}$, which implies that $N_{\fp}=0$ for all $\fp \ne \fm$ and $N$ is a finite length $R/I$-module. That is, $a \in I: \fm^k$ for $k \gg 0$. Note that $H^t_{\fm}(R) \cong H^0_{\fm}(H^t_I(R))$. Therefore $\overline{a} \in H^t_{\fm}(R)$. Since $R$ is $F$-injective, we have $\overline{a}=0$. Hence the injectivity of $R/I \to H^t_I(R)$ shows that $a \in I$.
\end{proof}

\begin{corollary}
\label{Frobeniuslocal}
Let $(R,\fm,k)$ be an $F$-finite local ring and let $I=(x,x_2,\ldots,x_t)$ be an ideal of $R$ which is generated by a regular sequence. Assume that $R/xR$ is $F$-injective. Then the Frobenius actions on both $H^t_I(R)$ and $H^t_{\frak m}(R)$ are injective.
\end{corollary}

\begin{proof}
By Proposition \ref{P1.13}, the ideal $(\overline{x_2},\ldots,\overline{x_t}) \subseteq R/xR$ is Frobenius closed. We prove that $I=(x,x_2,\ldots,x_d)$ is Frobenius closed. For this, let $u^q \in (x^q,x_2^q,\ldots,x_d^q)$ for $u \in R$ and $q=p^e \gg 0$. Then mapping this relation to the quotient ring $R/xR$ and since $\overline{x_2},\ldots,\overline{x_d}$ forms a system of parameters of $R/xR$, we have $u \in xR+(x_2,\ldots,x_d)$ by the assumption that $R/xR$ is $F$-injective, proving that $I$ is Frobenius closed. By Lemma \ref{L1.6}, the Frobenius power $I^{[p^e]}$ is also Frobenius closed.

We consider the commutative diagram
$$
\begin{CD}
R/I @>>> R/I^{[p]} @>>> R/I^{[p^2]} @>>> \cdots \\
@VFVV @VFVV @VFVV \\
R/I^{[p]} @>>> R/I^{[p^2]} @>>> R/I^{[p^3]} @>>> \cdots
\end{CD}
$$
The vertical map is the Frobenius map and the direct limit of the horizontal direction is the local cohomology $H^t_{I}(R)$ and the Frobenius action on $H^t_{I}(R)$ is injective. It is clear that a regular sequence is a filter regular sequence and by Nagel-Schenzel isomorphism:
$$
H^t_{\fm}(R) \cong H^0_{\fm}(H^t_{I}(R)),
$$
the Frobenius action on $H^t_{\fm}(R)$ is injective.
\end{proof}

\begin{remark}
Assume that $R$ is a weakly normal Noetherian ring of characteristic $p$. Then we can show that every principal ideal generated by a regular element $x \in R$ is Frobenius closed. To see this, let $y \in (x)^F$. Then $y^q \in (x^q)$ for some $q=p^e$. Hence we have
$$
\displaystyle \big(\big(\frac{y}{x}\big)^{p^{e-1}}\big)^p=\big(\frac{y}{x}\big)^{p^e} \in R.
$$
Considering this relation to belong to the total ring of fractions of $R$, we must have $\displaystyle \big(\frac{y}{x}\big)^{p^{e-1}} \in R$ by the definition of weak normality. That is, $y^{p^{e-1}} \in (x^{p^{e-1}})$. By induction on $e \ge 0$, it follows that $y \in (x)$, proving that the principal ideal $(x)$ is Frobenius closed. If $(R,\fm,k)$ is an $F$-finite $F$-injective local ring, then $R$ is weakly normal (cf. \cite[Theorem 4.7]{Sch09}).
\end{remark}

Recall that an $F$-injective local ring is reduced by Lemma \ref{L1.11}.

\begin{corollary}
\label{C1.15}
Let $(R,\fm,k)$ be an $F$-finite $F$-injective local ring. Let $x_1,\ldots,x_t$ be a regular sequence of $R$. Then $x_1,\ldots,x_t$ is a regular sequence of $R^{1/q}/R$.
\end{corollary}

\begin{proof}
For $i = 0,\ldots,t$, let $I_i=(x_1, \ldots, x_i)$. By Proposition \ref{P1.13}, $I_i$ is Frobenius closed for all $i \le t$. Thus we have the following short exact sequence
$$
0 \to R/I_iR \to R^{1/q}/I_iR^{1/q} \to (R^{1/q}/R)/I_i(R^{1/q}/R) \to 0.$$
For each $i = 1, \ldots, t$ we consider the following commutative diagram
$$
\begin{CD}
0 @>>> R/I_{i-1}R @>>> R^{1/q}/I_{i-1}R^{1/q} @>>> (R^{1/q}/R)/I_{i-1}(R^{1/q}/R) @>>>0\\
@. @VVx_iV @VVx_iV @VVx_iV \\
0 @>>> R/I_{i-1}R @>>> R^{1/q}/I_{i-1}R^{1/q} @>>> (R^{1/q}/R)/I_{i-1}(R^{1/q}/R) @>>> 0.
\end{CD}
$$
Since $x_1, \ldots , x_t$ is a regular sequence of both $R$ and $R^{1/q}$ we have the following exact sequence
$$0 \to 0:_{(R^{1/q}/R)/I_{i-1}(R^{1/q}/R)}x_i \to R/I_iR \to R^{1/q}/I_iR^{1/q} \to (R^{1/q}/R)/I_i(R^{1/q}/R) \to 0.$$
Therefore $0:_{(R^{1/q}/R)/I_{i-1}(R^{1/q}/R)}x_i = 0$ for all $i = 1, \ldots, t$. Thus $x_1, \ldots,x_t$ is a regular sequence of $R^{1/q}/R$.
\end{proof}

We need the following corollary in the sequel.

\begin{corollary}\label{C1.16}
Let $(R,\fm,k)$ be an $F$-finite $F$-injective local ring. Let $x_1,\ldots,x_t$ be a filter regular sequence of $R$. Then $x_1,\ldots,x_t$ is a filter regular sequence of $R^{1/q}/R$.
\end{corollary}

\begin{proof}
$R_{\fp}$ is a reduced $F$-finite $F$-injective local ring for $\fp \in \Spec R$ by Lemma \ref{L1.11}. The corollary follows from Lemma \ref{filter} together with Corollary \ref{C1.15}.
\end{proof}

\section{Generalized Cohen-Macaulay rings}
\label{sec3}

Let us recall the definition of generalized Cohen-Macaulay modules. Let $\ell_R(M)$ denote the length of an $R$-module $M$. Let $M$ be a finitely generated module over a local ring $(R,\fm,k)$ and let $\fq$ be a parameter ideal of $M$. We denote by $e(\fq,M)$ the multiplicity of $M$ with respect to $\fq$ (cf. \cite{BH98} for details).

\begin{definition}
\label{nCM1}
Let $M$ be a finitely generated module over a Noetherian local ring $(R,\fm,k)$ such that $d=\dim M>0$. Then $M$ is called \textit{generalized Cohen-Macaulay}, if the difference
$$
\ell_R(M/\fq M)-e(\fq,M)
$$
is bounded above, where $\fq$ ranges over the set of all parameter ideals of $M$.
\end{definition}

The following characterization of generalized Cohen-Macaulay modules plays the key role in this section.

\begin{theorem}
\label{charac gCM}
$M$ is generalized Cohen-Macaulay if and only if $H^i_{\fm}(M)$ is a finitely generated $R$-module for all $i < d=\dim M$.
\end{theorem}

\begin{remark}
Let the notation be as in Definition \ref{nCM1}.
\begin{enumerate}
\item
Under mild conditions of the base ring, $M$ is generalized Cohen-Macaulay if and only if the non-Cohen-Macaulay locus is isolated, and if and only if
every system of parameters forms a filter regular sequence (cf. \cite{CST78}).

\item
Let $M$ be a generalized Cohen-Macaulay $R$-module over $(R,\fm,k)$ such that $d=\dim M>0$. Then
$$
\ell_R(M/\fq M)-e(\fq, M) \le \sum_{i=0}^{d-1} \binom{d-1}{i} \ell_R(H^i_{\fm}(M))
$$
for every parameter ideal $\fq$ of $M$.
\end{enumerate}
\end{remark}

\subsection{Buchsbaum rings and standard sequences}

\begin{definition}[cf. \cite{Tr86}]
\label{nCM2}
Let $M$ be a finitely generated module over a Noetherian local ring $(R,\fm,k)$ such that $d=\dim M>0$. A parameter ideal $\fq$ of $M$ is called \textit{standard} if
$$
\ell_R(M/\fq M)-e(\fq, M) = \sum_{i=0}^{d-1} \binom{d-1}{i} \ell_R(H^i_{\fm}(M)).
$$
We say that $M$ is \textit{Buchsbaum}, if every parameter ideal of $M$ is standard.
\end{definition}

\begin{remark}
\label{R4.4}
Let $M$ be a generalized Cohen-Macaulay module over $(R,\fm,k)$ such that $d=\dim M>0$ and let $n \in \mathbb{N}$ be a positive integer such that $\fm^n \, H^i_{\fm}(M) = 0$ for all $i<d$. Then every parameter element $x \in \fm^{2n}$ of $M$ admits the splitting property, i.e., $H^i_{\fm}(M/xM) \cong H^i_{\fm}(M) \oplus H^{i+1}_{\fm}(M)$ for all $i<d-1$. Furthermore, every parameter ideal contained in $\fm^{2n}$ is standard (cf. \cite{CQ11}).
\end{remark}

We will use the following characterization of standard parameter ideal as its definition.

\begin{theorem}
\label{charac standard}
A parameter ideal $\fq=(x_1,\ldots,x_d)$ of $M$ is standard if and only if for every $i+j<d$, we have the equality:
$$
\fq \, H^i_{\fm}\Big(\dfrac{M}{(x_1, \ldots, x_j)M}\Big)=0.
$$
\end{theorem}

\begin{proof}
See \cite[Theorem 2.5]{Tr86}.
\end{proof}

We will generalize the following proposition in a more general context. The first assertion appeared in \cite[Proposition 1.4]{Su81}, and the second one is well-known in the theory of Buchsbaum rings.

\begin{proposition}
\label{length Koszul}
Let $M$ be a generalized Cohen-Macaulay module over $(R,\fm,k)$ such that $d=\dim M>0$ and let $\fq=(x_1,\ldots, x_d)$ be a parameter ideal of $M$. Then for all $s \le d$ we have
$$
\ell_R(H^i(x_1, \ldots,x_s; M)) \le \sum_{j=0}^{i} \binom{s}{i-j} \ell_R(H^j_{\fm}(M))
$$
for all $i<s$. Furthermore, $\fq$
is standard if and only if
$$
\ell_R(H^i(\fq; M)) = \sum_{j=0}^{i} \binom{d}{i-j} \ell_R(H^j_{\fm}(M))
$$
for all $i<d$.
\end{proposition}

The advantage of using characterizations of generalized Cohen-Macaulay modules and standard parameter ideals via local cohomology as in Theorem \ref{charac gCM} and Theorem \ref{charac standard} is that it allows us to consider certain problems in a more general context.

\begin{definition}(cf. \cite[Definition 9.1.3]{BS98})
\label{nCM3}
Let $M$ be a finitely generated module over a local ring $(R,\fm,k)$. The \textit{finiteness dimension $f_{\fm}(M)$ of $M$ with respect to} $\fm$ is defined as follows:
$$
f_{\fm}(M):=\mathrm{inf} \{i \,|\, H^i_{\fm}(M)~\text{is not finitely generated} \} \in \mathbb{Z}_{\ge 0} \cup \{\infty\}.
$$
\end{definition}

\begin{remark}
Assume that $\dim M=0$ or $M=0$ (recall that a trivial module has dimension $-1$). In this case, $H^i_{\fm}(M)$ is finitely generated for all $i$ and $f_{\fm}(M)$ is equal to $\infty$. It will be essential to know when the finiteness dimension is a positive integer. We mention the following result. Let $(R,\fm,k)$ be a local ring and let $M$ be a finitely generated $R$-module. If $d=\dim M >0$, then the local cohomology module $H^d_{\fm}(M)$ is not finitely generated. For the proof of this result, see \cite[Corollary 7.3.3]{BS98}.
\end{remark}

\begin{definition}
\label{nCM4}
Let $M$ be a finitely generated module over a local ring $(R,\fm,k)$ such that $t=f_{\fm}(M)<\infty$ and let $x_1,\ldots, x_s$, $s \le t$, be a filter regular sequence on $M$. Then we say that $x_1,\ldots, x_s$ is a \textit{standard sequence} of $M$ if
$$
(x_1,\ldots, x_s) \, H^i_{\fm}\Big(\dfrac{M}{(x_1,\ldots,x_j)M}\Big) = 0
$$
for all $i+j<s$.
\end{definition}

\begin{remark}\label{R5.10}
Let the notation be as in Definition \ref{nCM3}.
\begin{enumerate}
\item
Let $M$ be a finitely generated module over $(R,\fm,k)$ such that $d=\dim M>0$. Then $M$ is generalized Cohen-Macaulay if and only if $f_{\fm}(M) = d$.

\item
Assume that $f_{\fm}(M)<\infty$. By Grothendieck's finiteness theorem, we have
$$
f_{\fm}(M)=\min \{\depth(M_{\fp}) + \dim R/\fp \, |\, \fp \neq \fm \},
$$
provided that $R$ is a homomorphic image of a regular local ring (cf. \cite[Theorem 9.5.2]{BS98}).

\item By \cite[Lemma 2.9]{Q13a}, if $x_1, \ldots, x_s$ with $s \le f_{\fm}(M)$ is a filter regular sequence on $M$, then it is a filter regular sequence on $M$ in any order.
\end{enumerate}
\end{remark}

The following is useful in the sequel.

\begin{lemma}\label{length lc}
Let $M$ be a finitely generated module over $(R,\fm,k)$. Let $t=f_{\fm}(M)<\infty$ and let $x_1, \ldots, x_s$ with $s \le t$ be a filter regular sequence on $M$. Then for all $i+j < s$ we have
$$
\ell_R(H^i_{\fm}(M/(x_1, \ldots, x_j)M)) \le \sum_{k = i}^{i+j} \binom{j}{k-i} \ell_R(H^k_{\fm}(M)) \quad (\star).
$$
Moreover, $x_1, \ldots, x_s$ is a standard sequence of $M$ if and only if
$$
\ell_R \big(\dfrac{(x_1, \ldots, x_{s-1})M:_M x_s}{(x_1, \ldots, x_{s-1})M}\big)=
\sum_{k = 0}^{s-1} \binom{s-1}{k} \ell_R(H^k_{\fm}(M)).
$$
In this case, the above inequalities $(\star)$ become equalities.
\end{lemma}
\begin{proof} The short exact sequence
$$0 \to M/(0:_Mx_1) \to M \to M/x_1M \to 0$$
induces the exact sequence of local cohomology
$$
\cdots \to H^{k}_{\fm}(M) \to H^{k}_{\fm}(M/x_1M) \to H^{k+1}_{\fm}(M) \overset{x_1}{\to} H^{k+1}_{\fm}(M) \to \cdots.
$$
Therefore, for all $k<s-1$ we have
$$
\ell_R(H^{k}_{\fm}(M/x_1M)) \le \ell_R(H^{k}_{\fm}(M)) + \ell_R(H^{k+1}_{\fm}(M)),
$$
and the equality occurs if $x_1 \, H^k_{\fm}(M) = 0$ for all $k<s$. By induction we obtain the inequality
$$
\ell_R(H^i_{\fm}(M/(x_1, \ldots, x_j)M)) \le \sum_{k = i}^{i+j} \binom{j}{k-i} \ell_R(H^k_{\fm}(M))
$$
for all $i+j < s$, and the equality occurs if $x_1, \ldots, x_s$ is standard.\\
Conversely, suppose
$$
\ell_R \big(\dfrac{(x_1, \ldots, x_{s-1})M:_M x_s}{(x_1, \ldots, x_{s-1})M}\big)=
\sum_{k = 0}^{s-1} \binom{s-1}{k} \ell_R(H^k_{\fm}(M)).
$$
We prove $x_1, \ldots, x_s$ is standard by induction on $s$. If $s =1$, we have $0:_M x_1 = H^0_{\fm}(M)$. Hence $x_1 \, H^0_{\fm}(M) = 0$, and so $x_1$ is a standard element. For $s>1$, as above
$$
\ell_R \big(\dfrac{(x_1, \ldots, x_{s-1})M:_M x_s}{(x_1, \ldots, x_{s-1})M}\big) \le
\sum_{k = 0}^{s-2} \binom{s-2}{k} \ell_R(H^k_{\fm}(M/x_1M)).
$$
On the other hand, since $\ell (H^{k}_{\fm}(M/x_1M)) \le \ell(H^{k}_{\fm}(M)) + \ell(H^{k+1}_{\fm}(M))$ for all $s<k-1$ we have
$$\sum_{k = 0}^{s-2} \binom{s-2}{k} \ell_R(H^k_{\fm}(M/x_1M)) \le \sum_{k = 0}^{s-1} \binom{s-1}{k} \ell_R(H^k_{\fm}(M)).$$
Therefore, all inequalities become equalities. Hence $x_2, \ldots, x_s$ is a standard sequence of $M/x_1M$ by the inductive hypothesis, that is $(x_1, \ldots, x_s)\, H^i_{\fm}(M/(x_1, \ldots,x_j)M) = 0$ for all $i+j <s$ and $j \ge 1$. We need only to show that $(x_1, \ldots, x_s)\, H^i_{\fm}(M) = 0$ for all $i<s$. It follows from the short exact sequence
$$0 \to H^{i}_{\fm}(M) \to H^{i}_{\fm}(M/x_1M) \to H^{i+1}_{\fm}(M) \to 0$$
for all $i<s-1$, and the fact $(x_1, \ldots, x_s)\, H^i_{\fm}(M/x_1M) = 0$ for all $i<s-1$. The proof is complete.
\end{proof}

The following proposition is the reformulation of Proposition \ref{length Koszul} in our context. For convenience, we provide a detailed proof.

\begin{proposition}
\label{L2.10}
Let $M$ be a finitely generated module over $(R,\fm,k)$. Let $t=f_{\fm}(M)<\infty$ and let $x_1, \ldots, x_s$ with $s \le t$ be a filter regular sequence on $M$. Then
$$
\ell_R(H^i(x_1, \ldots, x_s; M)) \le \sum_{j=0}^{i} \binom{s}{i-j} \ell_R(H^j_{\fm}(M))
$$
for all $i<s$. Furthermore, $x_1, \ldots, x_s$ is a standard sequence of $M$ if and only if the above inequalities become equalities for all $i<s$.
\end{proposition}

\begin{proof}
We prove the first assertion by induction on $s$. The case $s=1$ is obvious. For $s>1$ we have the following exact sequence of Koszul cohomology
$$
\cdots \to H^{i-1}(x_1, \ldots, x_{s-1};M) \overset{\pm x_s}{ \to } H^{i-1}(x_1, \ldots, x_{s-1};M) \to H^{i}(x_1, \ldots, x_{s};M) \to H^{i}(x_1, \ldots, x_{s-1};M) \overset{\pm x_s}{ \to } \cdots.
$$
Hence for all $i<s-1$ we have
$$\ell_R(H^{i}(x_1, \ldots, x_{s};M)) \le \ell_R(H^{i-1}(x_1, \ldots, x_{s-1};M)) + \ell_R(H^{i}(x_1, \ldots, x_{s-1};M)),$$
and for $i=s-1$ we have
\begin{eqnarray*}
\ell_R(H^{s-1}(x_1, \ldots, x_{s};M)) &\le  & \ell_R(H^{s-2}(x_1, \ldots, x_{s-1};M)) + \ell_R \big(\dfrac{(x_1, \ldots, x_{s-1})M:_M x_s}{(x_1, \ldots, x_{s-1})M} \big)\\
&\le & \ell_R(H^{s-2}(x_1, \ldots, x_{s-1};M)) + \ell_R(H^0_{\fm}(M/(x_1, \ldots, x_{s-1})M)).
\end{eqnarray*}
By induction and Lemma \ref{length lc}, we can check that
$$
\ell_R(H^i(x_1, \ldots,x_s; M)) \le \sum_{j=0}^{i} \binom{s}{i-j} \ell_R(H^j_{\fm}(M))
$$
for all $i<s$.

For the second assertion, suppose that $x_1, \ldots, x_s$ is a standard sequence of $M$. For all $j \le s$, since $(x_1, \ldots, x_s)H^0_{\fm}(M/(x_1, \ldots, x_{j-1})M) = 0$, we have
$$
(x_1, \ldots, x_{j-1})M :_M x_j=\bigcup_{n \ge 1} (x_1, \ldots, x_{j-1})M :_M \fm^n.
$$
By \cite[Remark 2.8]{Tr86}, we have $x_1,\ldots,x_s$ is a $d$-sequence and thus,
$x_s H^i(x_1, \ldots, x_{s-1}; M) = 0$ for all $i< s-1$ (cf. \cite[Section 5]{HSV82}). Therefore, the above long exact sequence of Koszul cohomology yields
$$
\ell_R(H^{i}(x_1, \ldots, x_{s};M))=\ell_R(H^{i-1}(x_1, \ldots, x_{s-1};M)) + \ell_R(H^{i}(x_1, \ldots, x_{s-1};M)),
$$
for all $i < s-1$, and
\begin{eqnarray*}
\ell_R(H^{s-1}(x_1, \ldots, x_{s};M)) &=& \ell_R(H^{s-2}(x_1, \ldots, x_{s-1};M)) + \ell_R\big(\dfrac{(x_1, \ldots, x_{s-1})M:_M x_s}{(x_1, \ldots, x_{s-1})M} \big)\\
&= & \ell_R(H^{s-2}(x_1, \ldots, x_{s-1};M)) + \ell_R(H^0_{\fm}(M/(x_1,\ldots,x_{s-1})M))
\end{eqnarray*}
where we used the fact $x_s \, H^0_{\fm}(M/(x_1, \ldots, x_{s-1})M) = 0$ for the second equation. By induction and Lemma \ref{length lc}, we have
$$
\ell_R(H^i(x_1, \ldots,x_s; M)) = \sum_{j=0}^{i} \binom{s}{i-j} \ell_R(H^j_{\fm}(M))
$$
for all $i<s \le t$. Conversely, suppose
$$
\ell_R(H^i(x_1, \ldots,x_s; M)) = \sum_{j=0}^{i} \binom{s}{i-j} \ell_R(H^j_{\fm}(M))$$
for all $i<s$. As above we see that if
$$
\ell_R(H^{s-1}(x_1, \ldots,x_s; M)) = \sum_{j=0}^{s-1} \binom{s}{s-j-1} \ell_R(H^j_{\fm}(M)),
$$
then
$$
\ell_R \big(\dfrac{(x_1, \ldots, x_{s-1})M:_M x_s}{(x_1, \ldots, x_{s-1})M}\big) = \sum_{j = 0}^{s-1} \binom{s-1}{j} \ell_R(H^j_{\fm}(M)).
$$
Hence $x_1, \ldots,x_s$ is a standard sequence by Lemma \ref{length lc}. The proof is complete.
\end{proof}

\begin{corollary}\label{permutation standard}
  Let $M$ be a finitely generated module over $(R,\fm,k)$. Let $t=f_{\fm}(M)<\infty$, let $x_1, \ldots, x_s$, $s \le t$, be a filter regular sequence on $M$. Suppose $x_1, \ldots, x_s$ is a standard sequence. Then any permutation of $x_1, \ldots, x_s$ is also a standard sequence.
\end{corollary}
\begin{proof} Any permutation of $x_1, \ldots, x_s$ is a filter regular sequence by Remark \ref{R5.10} (3). The assertion follows from the fact Koszul cohomology is not depend on the order of sequence of elements.
\end{proof}

The following lemma is inspired by \cite[Proposition 3.2]{Tr86} in our context.

\begin{lemma}
\label{standardcri}
Let $M$ be a finitely generated module over $(R,\fm,k)$ and let $\fa \subseteq R$ be an $\fm$-primary ideal. Let $s \le f_{\fm}(M)$ be an integer. Then the following are equivalent.

\begin{enumerate}
\item
Every filter regular sequence of length $s$ in $\fa$ is standard on $M$.

\item
The ideal $\fa$ has a set of generators $S$ such that for every subset of $S$ with cardinality equal to $s$ forms a filter regular sequence on $M$ and this sequence is standard.
\end{enumerate}
\end{lemma}

\begin{proof}
We need only to show the direction $(2) \Rightarrow (1)$. We have to prove that every filter regular sequence $y_1, \ldots,y_s$ on $M$ is standard. If $s=1$, we have
$$
0:_M y_1 \supseteq 0:_M \fa = \bigcap_{x \in S} 0:_M x = \bigcup_{k\ge 1} 0:_M \fm^k \supseteq 0:_M y_1.
$$
Thus $y_1 \, H^0_{\fm}(M) = 0$, and so $y_1$ is a standard element. If $s>1$, we can find a set of generators $S'$ of $\fa$ such that $x_1, \ldots, x_{s-1},z$ and $y_1, \ldots, y_{s-1},z$ are filter regular sequences on $M$ for all $z \in S'$ and $\{x_1, \ldots, x_{s-1}\} \subseteq S$ by the same method of \cite[Chapter I, Proposition 1.9]{SV86}. Since every subset of $S$ which consists of $s$ elements forms a standard sequence on $M$, we have $(x_1, \ldots, x_{s-1},x) \, H^0_{\fm}(M/(x_1, \ldots, x_{s-1})M) = 0$ for all $x \in S \setminus \{x_1, \ldots, x_{s-1}\}$. Therefore,
\begin{eqnarray*}
(x_1, \ldots, x_{s-1})M :_M z & \supseteq & (x_1, \ldots, x_{s-1})M :_M \fa \\
& = & \bigcap_{x \in S} (x_1, \ldots, x_{s-1})M :_M x\\
& = & \bigcup_{k \ge 0} (x_1, \ldots, x_{s-1})M :_M \fm^k\\
& \supseteq & (x_1, \ldots, x_{s-1})M :_M z.
\end{eqnarray*}
Hence $(x_1, \ldots, x_{s-1})M :_M z = (x_1, \ldots, x_{s-1})M :_M \fm^k = (x_1, \ldots, x_{s-1})M :_M x_s$, so
$$
\ell_R\big(\dfrac{(x_1, \ldots, x_{s-1})M :_M z}{(x_1, \ldots, x_{s-1})M } \big) = \ell_R \big(\dfrac{(x_1, \ldots, x_{s-1})M :_M x_s}{(x_1, \ldots, x_{s-1})M } \big).$$
Since $x_1, \ldots, x_s$ is a standard sequence, $x_1, \ldots, x_{s-1},z$ is also a standard sequence by Lemma \ref{length lc}. By Corollary \ref{permutation standard}, $z, x_1, \ldots, x_{s-1}$ is a standard sequence on $M$. It is clear that $x_1, \ldots, x_{s-1}$ is a standard sequence on $M/zM$. Now by induction, we have $y_1, \ldots, y_{s-1}$ is a standard sequence on $M/zM$. Thus
\begin{eqnarray*}
\ell_R \big(\dfrac{(z, y_1, \ldots, y_{s-2})M :_M y_{s-1}}{(z,y_1, \ldots, y_{s-2})M } \big) &=& \sum_{j = 0}^{s-2} \binom{s-2}{j} \ell_R(H^j_{\fm}(M/zM)).\\
&=& \sum_{j= 0}^{s-1} \binom{s-1}{j} \ell_R(H^j_{\fm}(M)).
\end{eqnarray*}
Therefore, $z, y_1, \ldots, y_{s-1}$ is a standard sequence of $M$. Using Corollary \ref{permutation standard} we have $y_1, \ldots, y_{s-1}, z$ is a standard sequence on $M$. Since the elements $z \in S'$ generate $\fa$, we can show, by similar as above that $y_1, \ldots, y_s$ is a standard sequence of $M$. The proof is complete.
\end{proof}

\subsection{Frobenius closed parameter ideals II}
We prove the following lemma.

\begin{lemma}
\label{filterquotient}
Let $(R,\fm,k)$ be a reduced $F$-finite local ring of characteristic $p>0$ and let $x_1,\ldots,x_s$ be a filter regular sequence of $R^{1/p}/R$. Then it is also a filter regular sequence of $R^{1/q}/R$ for all $q=p^e$.
\end{lemma}

\begin{proof}
Using Lemma \ref{filter} (3) and applying the induction for the sequence
$$
0 \to R^{1/p}/R \to R^{1/q}/R \to R^{1/q}/R^{1/p} \to 0,
$$
we get the proof of the lemma.
\end{proof}

The following is the first main result of this section.

\begin{theorem}
\label{finitedim}
Let $(R,\fm,k)$ be a reduced $F$-finite local ring of characteristic $p>0$ with $f_{\fm}(R)=t$ and let $0<s \le t$ be an integer. Then the following are equivalent.

\begin{enumerate}
\item
The Frobenius action on $H^i_{\fm}(R)$ is injective for all $i \le s$.

\item
Every filter regular sequence $x_1,\ldots,x_s$ of both $R$ and $R^{1/p}/R$ generates a Frobenius closed ideal of $R$ and it is a standard sequence on $R$.

\item
Every filter regular sequence $x_1, \ldots, x_s$ of both $R$ and $R^{1/p}/R$ generates a Frobenius closed ideal of $R$.
\end{enumerate}
\end{theorem}

\begin{proof}
$(2) \Rightarrow (3)$ is clear and $(3) \Rightarrow (1)$ follows from Theorem \ref{T1.2}.

For $(1) \Rightarrow (2)$, let $x_1,\ldots,x_s$ with $s \le t$ be a filter regular sequence of both $R$ and $R^{1/p}/R$. We prove that $x_1,\ldots,x_s$ is standard and the ideal $I = (x_1,\ldots,x_s)$ is Frobenius closed. Let $n_0>0$ be an integer such that $\fm^{n_0}H^i_{\fm}(R) = 0$ for all $i<s$. For each $x \in \frak m$, choose $e \ge 1$ such that $q = p^e > n_0$. We have $F^e_*(xH^i_{\fm}(R))=x^q F^e_*(H^i_{\fm}(R)) = 0$, where $F_*$ is the natural Frobenius action on the local cohomology. By $F$-injectivity of $R$, we have $x \cdot H^i_{\fm}(R) = 0$ and so $\fm \cdot H^i_{\frak m}(R) = 0$  for all $i<s$. By \cite[Theorem 1.1, Corollary 4.1]{CQ11}, it is easy to see that $x_1^q,\ldots,x_s^q$ is a standard sequence of $R$ for all $q = p^e \ge 2$. It is equivalent to say that $x_1,\ldots,x_s$ is a standard sequence of $R^{1/q}$. Since $R$ is reduced and $F$-finite, we get a short exact sequence of finitely generated $R$-modules:
$$
0 \to R \to R^{1/q} \to R^{1/q}/R \to 0.
$$
Because the Frobenius action on $H^i_{\fm}(R)$ is injective for all $i \le s$, each induced homomorphism $H^i_{\fm}(R) \to H^i_{\fm}(R^{1/q})$ is injective for all $i \le s$. Thus we have short exact sequences
$$
0 \to H^i_{\fm}(R) \to H^i_{\fm}(R^{1/q}) \to H^i_{\fm}(R^{1/q}/R) \to 0
$$
for all $i < s$. Therefore,
$$
\ell_R(H^i_{\fm}(R^{1/q}))=\ell_R(H^i_{\fm}(R))+\ell_R(H^i_{\fm}(R^{1/q}/R))
$$
for all $i<s$. By Lemma \ref{filterquotient}, $x_1, \ldots, x_s$ is a filter regular sequence of $R^{1/q}/R$. Applying Proposition \ref{L2.10}, for all $i<s$ we have
\begin{eqnarray*}
\ell_R(H^i(I; R)) &\le &\sum_{j=0}^{i} \binom{s}{i-j} \ell_R(H^j_{\fm}(R)) \\
\ell_R(H^i(I; R^{1/q})) &= &\sum_{j=0}^{i} \binom{s}{i-j} \ell_R(H^j_{\fm}(R^{1/q})) \\
\ell_R(H^i(I; R^{1/q}/R)) &\le &\sum_{j=0}^{i} \binom{s}{i-j} \ell_R(H^j_{\fm}(R^{1/q}/R)),
\end{eqnarray*}
where the middle equation follows from the fact that $x_1,\ldots,x_s$ is a standard sequence of $R^{1/q}$. On the other hand, by applying Koszul cohomology to the sequence
$$
0 \to R \to R^{1/q} \to R^{1/q}/R \to 0,
$$
we have $\ell_R(H^i(I; R^{1/q})) \le \ell_R(H^i(I; R)) + \ell_R(H^i(I; R^{1/q}/R))$ for all $i<s$. Therefore,
\begin{eqnarray*}
\ell_R(H^i(I; R^{1/q})) &=& \ell_R(H^i(I; R)) + \ell_R(H^i(I; R^{1/q}/R)) \quad (\star) \\
\ell_R(H^i(I; R)) &= &\sum_{j=0}^{i} \binom{s}{i-j} \ell_R(H^j_{\fm}(R)) \quad (\star \star)
\end{eqnarray*}
for all $i<s$. By $(\star \star)$ and Proposition \ref{L2.10}, it follows that $x_1, \ldots, x_s$ is a standard sequence of $R$. Applying $(\star)$ to the exact sequence of Koszul cohomology:
$$
0 \to H^0(I; R) \to H^0(I; R^{1/q}) \to H^0(I; R^{1/q}/R) \to \cdots
$$
$$
\to H^{s-1}(I; R^{1/q}/R) \to H^{s}(I; R) \to H^{s}(I; R^{1/q}) \to \cdots$$
we have short exact sequences
$$
0 \to H^i(I; R) \to H^i(I; R^{1/q}) \to H^i(I; R^{1/q}/R) \to 0
$$
for all $i<s$ and the injection
$$
0 \to H^s(I; R) \cong R/I \to H^s(I; R^{1/q}) \cong R^{1/q}/IR^{1/q}
$$
for $q=p^e$ with $e>0$. Thus, $I$ is Frobenius closed.
\end{proof}

\begin{remark}
The condition that $x_1,\ldots,x_s$ is a filter regular sequence of both $R$ and $R^{1/p}/R$ is necessary. To see this, take $A$ to be a reduced and $F$-finite Cohen-Macaulay local ring which is not $F$-injective. So there exists a parameter ideal $(x_1,\ldots,x_s)$ of $A$ that is not Frobenius closed with $s=\dim A$. Let $R:=A[[x]]$. Then $R$ is Cohen-Macaulay and so $f_{\fm}(R)=\dim R=\dim A+1$. The assumption on the injectivity of the Frobenius is clear, since $H^i_{\fm}(R) = 0$ for all $i \le s$. However, $x_1,\ldots,x_s$ is a regular sequence of $R$ of length $s$, which generates a non-Frobenius closed ideal of $R$.
\end{remark}

The following theorem is the second main result of this section, which can be seen as a generalization of Proposition \ref{P1.13}.

\begin{theorem}
\label{maintheoremA}
Let $(R,\fm,k)$ be an $F$-injective local ring with $f_{\fm}(R)=t$. Then every filter regular sequence of length at most $t$ is a standard sequence and the ideal generated by it is Frobenius closed.
\end{theorem}

\begin{proof}
If $\dim R=0$, then $R$ is a field by assumption and there is nothing to prove. Therefore, we may assume that $\dim R>0$ and hence $t<\infty$. It is clear that $H^0_{\fm}(R) = 0$, so $\depth R>0$. Let $x_1,\ldots,x_s$ with $s \le t$ be a filter regular sequence of $R$. Using $\Gamma$-construction and taking a sufficiently small $\Gamma$, there is a chain of faithfully flat extensions of local rings of the same Krull dimension:
$$
R \to \widehat{R}^{\Gamma} \to S:=\widehat{\widehat{R}^{\Gamma}},
$$
where $\widehat{R}^{\Gamma}$ is $F$-finite and $F$-injective. Then we have
$$
\big((x_1,\ldots,x_{i-1}):_R x_i)/(x_1,\ldots,x_{i-1})\big) \otimes_R S \cong ((x_1,\ldots,x_{i-1}):_Sx_i)/(x_1,\ldots,x_{i-1})
$$
and therefore, $x_1,\ldots,x_s$ is a filter regular sequence on $R$ if and only if so is on $S$. Likewise, since local cohomology commutes with flat base change, $x_1,\ldots,x_s$ is a standard sequence on $R$ if and only if so is on $S$. Finally, the ideal $(x_1,\ldots,x_s)$ is Frobenius closed if and only if so is $(x_1,\ldots,x_s)S$. Hence, we may assume that $R$ is an $F$-finite $F$-injective complete local ring (the fact that $R$ is reduced follows from Lemma \ref{L1.11}). Moreover by Corollary \ref{C1.16}, $x_1, \ldots, x_s$ is a filter regular sequence of $R^{1/q}/R$ for all $q = p^e$. The theorem now follows from Theorem \ref{finitedim}.
\end{proof}

We recover the following corollary which is the main result of Ma \cite{M15}, giving an answer to Takagi's question  (cf. \cite[Open problem A.3]{KS11}). It is noted that the proof of \cite{M15} relies on \cite{GO83}.

\begin{corollary}
\label{Buchsbaum}
Let $(R,\fm,k)$ be a generalized Cohen-Macaulay $F$-injective local ring. Then $R$ is Buchsbaum.
\end{corollary}

\begin{proof}
Under the stated assumption, every system of parameters of $R$ is a filter regular sequence by Lemma \ref{filter}. The corollary follows from Definition \ref{nCM2} and Theorem \ref{maintheoremA}.
\end{proof}

\begin{remark}
Ma proved the following result in \cite{M15}. Let $(R,\fm,k)$ be a generalized Cohen-Macaulay ring of characteristic $p>0$. Then $R$ is $F$-injective if and only if every parameter ideal is Frobenius closed.
\end{remark}

\subsection{Frobenius closed parameter ideals III}
Let $(R,\fm,k)$ be a local ring of characteristic $p>0$. Set $d=\dim R$. In order to show that $R$ is $F$-injective, it is necessary to consider all local cohomology modules, while it suffices to consider local cohomology modules except the one of top degree to check the Buchsbaumness on $R$. It thus seems natural to pose the following question.

\begin{question}\label{gtakagi}
Let $(R, \fm,k)$ be a reduced $F$-finite generalized Cohen-Macaulay local ring. Suppose that the Frobenius acts on $H^i_{\fm}(R)$ injectively for all $i<d$. Then is $R$ Buchsbaum?
\end{question}

In the same spirits of the method in the proof of \cite[Chapter I, Proposition 1.9]{SV86}, we have the following lemma.

\begin{lemma}
\label{basisfilter}
Let $(R,\fm,k)$ be a local ring and with $\fa$ an $\fm$-primary ideal. Let $M_1, \ldots, M_k$ be finitely generated $R$-modules. Then we can find a set of generators $S$ of $\fa$ such that for any subset of $S$ forms a filter regular sequence on all of $M_1,\ldots,M_k$ in any order.
\end{lemma}

As observed from the remark after Theorem \ref{finitedim}, it is not true that any filter regular sequence of $R$ of length $s$ generates a Frobenius closed ideal under the assumption that the Frobenius action on $H^i_{\fm}(R)$ is injective for all $i \le s$. However, we have the following theorem.

\begin{theorem}
\label{injec imply stand}
Let $(R,\fm,k)$ be a reduced $F$-finite local ring with $f_{\fm}(R)=t \ge 1$ and let $s \le t$ be a positive integer. Suppose that the Frobenius action is injective on $H^i_{\fm}(R)$ for all $i<s$. Then every filter regular sequence of length $s$ of $R$ is standard.
\end{theorem}

\begin{proof}
Put $n=\dim_k \fm/\fm^2$. By Lemma \ref{basisfilter}, we can find a set of generators $S$ of $\fm$ such that every subset of $S$ which consists of $s$ elements, forms a filter regular sequence of both $R$ and $R^{1/p}/R$. Take any subset $\{x_1, \ldots, x_s\}$ of $S$. By Lemma \ref{standardcri}, we need only to show that the sequence $x_1, \ldots, x_s$ is standard. By Theorem \ref{finitedim}, we have that $(x_1,\ldots,x_i)$ is Frobenius closed for all $i<s$. It is enough to prove that $x_1,\ldots,x_s$ is a $d$-sequence. The following argument is inspired by \cite[Proposition 3.3]{M15}. We want to prove the equality
$$
(x_1, \ldots, x_{i-1}):_R x_ix_j = (x_1, \ldots, x_{i-1}):_R x_j
$$
for all $i \le j \le s$. One inclusion is obvious. For the other inclusion, let $y \in (x_1, \ldots, x_{i-1}):_R x_ix_j$. Notice that there is a positive integer $N$ such that every filter regular sequence of length at most $f_{\fm}(R)$ contained in $\fm^N$ is a $d$-sequence. Now for $q> N$, we have
\begin{eqnarray*}
&\,& yx_ix_j \in (x_1, \ldots, x_{i-1}) \\
&\Rightarrow & y^qx_i^q x_j^q \in (x_1, \ldots, x_{i-1})^{[q]}\\
&\Rightarrow & y^q x_j^q \in (x_1, \ldots, x_{i-1})^{[q]}\\
&\Rightarrow & y x_j \in (x_1, \ldots, x_{i-1})\\
&\Rightarrow & y \in (x_1, \ldots, x_{i-1}):_Rx_j,
\end{eqnarray*}
where the third line follows from the fact that $x_1^q, \ldots, x_s^q$ is a $d$-sequence (because $q > N $), and the fourth line from the fact that $(x_1, \ldots, x_{i-1})$ is Frobenius closed. The proof is complete.
\end{proof}

We obtain the following corollary as an affirmative answer to Question \ref{gtakagi}. It should be noted that Bhatt, Ma and Schwede \cite{BMS16} prove the same result by a different method.

\begin{corollary}
\label{takagi}
Let $(R,\fm,k)$ be a reduced $F$-finite generalized Cohen-Macaulay local ring with $d=\dim R$. Suppose the Frobenius action on $H^i_{\fm}(R)$ is injective for all $i<d$. Then $R$ is Buchsbaum.
\end{corollary}

\section{$F$-injective, $F$-pure and stably $FH$-finite rings}
\label{sec4}

\begin{definition}
Let $R$ be a Noetherian ring of characteristic $p>0$. Then $R$ is said to be \textit{$F$-pure} (resp. \textit{$F$-split}), if the Frobenius endomorphism $R \to R$ is pure (resp. split).
\end{definition}

It is easy to see that $F$-split rings are $F$-pure. If $R$ is an $F$-finite ring, then $R$ is $F$-pure if and only if $R$ is $F$-split. We will consider $F$-pure rings, because they behave better than $F$-split rings in the non $F$-finite case. If $R$ is $F$-pure, then every ideal is Frobenius closed, and the converse holds true under a mild condition (cf. \cite{Hoc77}).

\begin{definition}
Let $M$ be an $R$-module with a Frobenius action $F$. A submodule $N$ of $M$ is called \textit{F-compatible} if $F(N) \subseteq N$.
\end{definition}

In \cite{M14}, Ma showed that the local cohomology modules of $F$-pure local rings satisfy certain interesting conditions originally studied in \cite{EH08}.

\begin{definition}
We say that an $R$-module $M$ with a Frobenius action $F$ is \textit{anti-nilpotent}, if for any $F$-compatible submodule $N$, the induced Frobenius action of $F$ on $M/N$ is injective. We say that $(R,\fm,k)$ is \textit{stably $FH$-finite}, if the local cohomology $H^i_{\fm}(R)$ are anti-nilpotent for all $i \ge 0$.
\end{definition}

For later use, we prove the following lemma. We consider $M$ as an $R\{F\}$-module via the Frobenius action $F$.

\begin{lemma}
\label{antilemma}
Let $R$ be a Noetherian ring of characteristic $p>0$. Then
\begin{enumerate}
\item
Let $0 \to L\to M \to N \to 0$ be a short exact sequence of $R\{F\}$-modules. Then $M$ is anti-nilpotent if and only if so are $L$ and $N$.

\item
Let $L \to M \xrightarrow{\alpha} N$ be an exact sequence of $R\{F\}$-modules such that $L$ is anti-nilpotent and $F$ acts injectively on $N$. Then $F$ acts injectively on $M$.
\end{enumerate}
\end{lemma}

\begin{proof}
$\rm(1)$ is clear.

$\rm(2)$ We have a short exact sequence:
$$
0\rightarrow \ker(\alpha) \to M \rightarrow \im(\alpha)\rightarrow 0.
$$
Then $F$ acts injectively on $\ker(\alpha)$, since $\ker(\alpha)$ is an $R\{F\}$-subquotient of $L$ and $L$ is anti-nilpotent. Moreover, $F$ acts injectively on $\im(\alpha)$, since $\im(\alpha)$ is an $R\{F\}$-submodule of $N$ and $F$ acts injectively on $N$. Now $F$ acts injectively on $M$.
\end{proof}

It is clear that a stably $FH$-finite local ring is $F$-injective. We need the following result in the sequel (cf. \cite[Theorem 2.3 and Theorem 3.8]{M14}).

\begin{theorem}
\label{anti-nilpotent}
Let $(R,\fm,k)$ be an $F$-pure local ring. Then the local cohomology modules $H^i_{\fm}(R)$ are anti-nilpotent for all $i \ge 0$ i.e., $R$ is stably $FH$-finite.
\end{theorem}

We prove the main result of this section, which is a variation of \cite[Proposition 4.8]{Sch09}.

\begin{theorem}
\label{geometry}
Let $(R,\fm,k)$ be a local ring of characteristic $p$. Suppose there exist ideals $I,J$ of $R$ such that $R/(I+J)$ is $F$-pure, $R/I$ and $R/J$ are $F$-injective. Then $R/(I \cap J)$ is $F$-injective.
\end{theorem}

\begin{proof}
There is a short exact sequence
$$
0 \to R/(I \cap J) \xrightarrow{p_1} R/I \oplus R/J \xrightarrow{p_2} R/(I+J) \to 0,
$$
where $p_1(a)=(a,-a)$ and $p_2(a,b)=a+b$ and this is compatible with the Frobenius. Taking local cohomology, we get
$$
\begin{CD}
@>>> H^i_{\fm}(R/(I+J)) @>h>> H^{i+1}_{\fm}(R/(I \cap J)) @>g>> H^{i+1}_{\fm}(R/I) \oplus H^{i+1}_{\fm}(R/J) @>>> \\
@. @VF_3VV @VF_1VV @VF_2VV \\
@>>> H^i_{\fm}(R/(I+J)) @>h>> H^{i+1}_{\fm}(R/(I \cap J)) @>g>> H^{i+1}_{\fm}(R/I) \oplus H^{i+1}_{\fm}(R/J) @>>> \\
\end{CD}
$$
By assumption, $R/(I+J)$ is $F$-pure and its local cohomology is anti-nilpotent by Theorem \ref{anti-nilpotent}. That is, the Frobenius action on $\im(h)$ induced by $F_3$ is injective. Now we conclude that $R/(I \cap J)$ is $F$-injective by Lemma \ref{antilemma} (2).
\end{proof}

With the same idea of the proof of Theorem \ref{geometry} together with Lemma \ref{antilemma} (1), we have the following result.

\begin{theorem}
\label{stably FH}
Let $(R,\fm,k)$ be a local ring of characteristic $p$. Suppose that there exist ideals $I,J$ of $R$ such that $R/I$, $R/J$ and $R/(I+J)$ are stably $FH$-finite. Then $R/(I \cap J)$ is stably $FH$-finite.
\end{theorem}

Theorem \ref{geometry} and Theorem \ref{stably FH} are useful in the construction of examples of non-Cohen-Macaulay local rings which are $F$-injective and not $F$-pure. To the best of authors' knowledge, examples of such type do not abound in literatures. We examine the construction of such local rings in the next section.

\section{Examples}
\label{sec5}
\subsection{Patching $F$-injective closed subschemes}

The aim of this section is to give an explicit example of $F$-injective ring with a parameter ideal that is not Frobenius closed. To do that, we need examples of $F$-injective rings that are neither $F$-pure nor generalized Cohen-Macaulay (with ``patching'' from the previous section). We start with a well-known example of Fedder \cite{F83} and Singh \cite{Sin99}.

\begin{lemma}
\label{singh}
Let $K$ be a perfect field of characteristic $p>0$ and let
$$
R:=K[[U,V,Y,Z]]/(UV,UZ,Z(V-Y^2)).
$$
Then $R$ is stably $FH$-finite (so $F$-injective).
\end{lemma}

\begin{proof}
Let $S = K[[U,V,Y,Z]]$. Note that
$$
(UV,UZ,Z(V-Y^2))=(U, V-Y^2) \cap (Z, U) \cap (Z,V) = (U, V-Y^2) \cap (Z, UV)
$$
Let $I = (U, V-Y^2)$ and $J = (Z, UV)$, so $I + J = (U, V-Y^2, Z)$. Therefore $S/I$ and $S/(I+J)$ are regular rings and $S/J$ is $F$-pure, since $J$ is a square-free ideal (cf. \cite[Proposition 5.38]{HR76}). By Theorems \ref{anti-nilpotent} and \ref{stably FH} we have $R \cong S/(I \cap J)$ is stably $FH$-finite. \end{proof}

\begin{remark}
Let $R$ be the ring as in Lemma \ref{singh}. We have the following.
\begin{enumerate}
\item
By \cite[Corollary 4.14]{HMS14}, $F$-purity deforms $F$-injectivity. We can prove that $R$ is $F$-injective in the following way. Let $u,v,y$ and $z$ denote the image of $U, V, Y$ and $Z$ in $R$ (and its quotients), respectively. Then $y$ is a regular element of $R$ and $R/(y) \cong K[[U, V, Z]]/(UV, UZ, VZ)$ is an $F$-pure ring by \cite[Proposition 5.38]{HR76}. Therefore $R$ is $F$-injective.

\item
By \cite[Example 3.2]{Sin99}, $R$ is not $F$-pure. However, we can check that $R$ is Cohen-Macaulay so every parameter ideal of $R$ is Frobenius closed.
\end{enumerate}
\end{remark}

We now use Theorems \ref{geometry} and \ref{stably FH} to construct $F$-injective rings that are neither $F$-pure nor generalized Cohen-Macaulay. We will produce all necessary computations.

\begin{example}
\label{E6.1}
Let $K$ be a perfect field of characteristic $p>0$ and let
$$
R:=K[[U,V,Y,Z,T]]/(T) \cap (UV,UZ,Z(V-Y^2)).
$$
This is an $F$-finite non-equidimensional local ring of dimension 4. Let $u,v,y, z$ and $t$ denote the image of $U, V, Y, Z$ and $T$ in $R$ (or the quotient ring of $R$), respectively. Let $S = K[[U,V,Y,Z,T]]$, $I = (T)$ and $J = (UV,UZ,Z(V-Y^2))$. By the similar method as in the proof of Lemma \ref{singh}, we find that $R$ is stably $FH$-finite, so is $F$-injective.

On the other hand, the only associated prime ideal $\fp \in \Spec R$ such that $R/\fp=4$ is $(t)$. Let $a:=y^2(u^2-z^4)$. Then $a$ is a parameter element of $R$. Since $a \in (u,z)$, we see that $a \in R$ is a zero divisor. Now we prove that the ideal $(a) \subseteq R$ is not Frobenius closed. We have $zvt=zy^2t$ and $y^{2p}u^{2p}=y^{2p}z^{4p}$ in the ring $R/(a)^{[p]}$ and it follows that
$$
(y^3z^4t)^p=y^{3p}z^{4p}t^p=y^{3p-2}y^2z^{4p}t^p
=y^{3p-2}vz^{4p}t^p=y^{3p-2}u^{2p}vt^p=0
$$
in $R/(a)^{[p]}$. That is, we have $(y^3z^4t)^p \in (a)^{[p]}$. Next consider the equation
$$
Y^3Z^4T=A(Y^2(U^2-Z^4))+B(TUV)+C(TUZ)+D(TZ(V-Y^2))
$$
in $K[[U,V,Y,Z,T]]$. We have $A = TA'$. Then dividing this equation out by $T$, we simply get
$$
Y^3Z^4=A'(Y^2(U^2-Z^4))+B(UV)+C(UZ)+D(Z(V-Y^2)).
$$
Taking this equation modulo $(U,V)$, we have $A'=-Y$ and $Y^3U=B(V)+C(Z)+D'(Z(V-Y^2))$, where $D=UD'$. Thus $Y^3U \in (V,Z)$ and this is impossible. Hence $y^3z^4t \notin (a)$ and $(a)$ is not Frobenius closed. Thus $R$ is $F$-injective, but not $F$-pure. Moreover, a generalized Cohen-Macaulay local ring is equidimensional, so $R$ is not generalized Cohen-Macaulay.
\end{example}

Based on the previous example, we can construct a local ring that is equidimensional, $F$-injective but not generalized Cohen-Macaulay and not $F$-pure.

\begin{example}\label{E6.2}
Let $K$ be a perfect field of characteristic $p>0$ and let
$$
R:=K[[U,V,Y,Z,T,S]]/(T, S) \cap (UV,UZ,Z(V-Y^2)).
$$
This is an equidimensional local ring of dimension 4. As in the the previous example, we find that $R$ is stably $FH$-finite, so it is $F$-injective. To show that $R$ is not $F$-pure, consider
$$
Y^3Z^4T=A(Y^2(U^2-Z^4))+B(TUV)+C(TUZ)+D(TZ(V-Y^2))
$$
$$
+E(SUV)+F(SUZ)+G(SZ(V-Y^2))
$$
in the ring $K[[U,V,Y,Z,T,S]]$. Taking this equation modulo $(S)$, we get
$$
Y^3Z^4T=A(Y^2(U^2-Z^4))+B(TUV)+C(TUZ)+D(TZ(V-Y^2))
$$
in $K[[U,V,Y,Z,T]]$. As in the previous example, we see that the ideal $(y^2(u^2-z^4))$ is not Frobenius closed and thus, $R$ is not $F$-pure. Set $I=(T, S)$ and $J=(UV,UZ,Z(V-Y^2))$. Moreover, both $R/I$ and $R/J$ are Cohen-Macaulay rings and we have $\dim R/I=\dim R/J=4$
and $\dim R/(I+J) = 2$. Applying the local cohomology for the short exact sequence
$$
0 \to R \to R/I \oplus R/J \to R/(I+J) \to 0,
$$
we get the following exact sequence
$$
0 \to H^2_{\fm}(R/(I+J)) \to H^3_{\fm}(R) \to H^3_{\fm}(R/I) \oplus H^3_{\fm}(R/J)=0.
$$
Thus $H^3_{\frak m}(R) \cong H^2_{\fm}(R/(I+J))$ does not have finite length by Grothendieck's non-vanishing theorem (cf. \cite[Corollary 7.3.3]{BS98}). So $R$ is not generalized Cohen-Macaulay.
\end{example}

The above examples prove the following main result of this section.

\begin{theorem}
\label{counterexample}
Let the ring $R$ and the notation be as in Example \ref{E6.1}. Then $R$ is $F$-injective and has a parameter ideal that is not Frobenius closed.
\end{theorem}

\begin{proof}
Let us construct a parameter ideal of $R$ that is not Frobenius closed. Since $(t)$ is the unique minimal associated prime of $R$ such that $\dim R/(t)=4$, the element $a=y^2(u^2-z^4)$ is a parameter element of $R$. Note that $a \in R$ is a zero divisor, because $a \in (u,z)$. Thus, we can extend it to a full system of parameters $a,x_2,x_3,x_4$ of $R$. By the Krull intersection theorem, we have
$$
(a) = \bigcap_{n \ge 1} (a,x_2^n,x_3^n,x_4^n).
$$
By choosing $n \gg 0$, we find that $b=y^3z^4t$ is not contained in the parameter ideal $(a,x_2^n,x_3^n,x_4^n)$. However, we have $b \in (a)^F \subseteq (a,x_2^n,x_3^n,x_4^n)^F$ as demonstrated in Example \ref{E6.1}. Hence $(a,x_2^n,x_3^n,x_4^n)$ is not Frobenius closed for $n \gg 0$.
\end{proof}

We have the following corollary.

\begin{corollary}
\label{F-example}
There exists an $F$-finite local ring $(R,\fm,k)$ of characteristic $p>0$ which is non-Cohen-Macaulay, $F$-injective, but not $F$-pure. Moreover, $R$ has a parameter ideal that is not Frobenius closed.
\end{corollary}

\begin{remark}
\begin{enumerate}
\item
We cannot put the ring $R$ of Example \ref{E6.2} into Theorem \ref{counterexample} to deduce the same conclusion, because $a \in R$ is a not parameter element in this case. In view of Theorem \ref{maintheoremA} and Corollary \ref{Buchsbaum}, it seems hard to construct an example of an $F$-injective local domain with a parameter ideal that is not Frobenius closed. We make a useful comment on the construction of Buchsbaum rings. By \cite{G80}, it is possible to construct a Buchsbaum ring $(R,\fm,k)$ with $d=\dim R$ such that $\ell_R(H^i_{\fm}(R))=s_i$, where $s_0,\ldots,s_{d-1}$ is any assigned sequence of non-negative integers.

\item
It is worth noting that Theorem \ref{counterexample} also claims that the result of Theorem \ref{maintheoremA} is optimal. Indeed, we see that $\dim R=4$ and $f_{\fm}(R) = \depth (R) = 3$. By the prime avoidance theorem, we can choose a filter regular sequence $x_1, x_2, x_3$ (so it is a regular sequence) such that $x_1, x_2, x_3, a$ form a system of parameters of $R$. Note that $x_1, x_2, x_3, a$ form a filter regular sequence. Using the Krull intersection theorem as in the proof of Theorem \ref{counterexample} and Lemma \ref{filter} (3), we can assume that $b \notin (x_1,x_2,x_3,a)$. Therefore, we find that $(x_1,x_2,x_3)$ is Frobenius closed by Theorem \ref{maintheoremA}, but $(x_1,x_2,x_3,a)$ is not Frobenius closed.
\end{enumerate}
\end{remark}

\subsection{Parameter $F$-closed rings}

We introduce a new class of $F$-singularities.

\begin{definition}
\label{parameterF-closed}
Let $(R,\fm,k)$ be a local ring of characteristic $p>0$. We say that $R$ is a \textit{parameter F-closed ring} if every ideal generated by a system of parameters of $R$ is Frobenius closed.
\end{definition}

We prove that the property of being parameter $F$-closed commutes with localization. It should be noted that we do not require the $F$-finiteness condition as in Lemma \ref{L1.11}.

\begin{proposition}
Let $(R,\fm,k)$ be a local ring of characteristic $p>0$. Then $R$ is parameter $F$-closed if and only if $R_{\fp}$ is parameter $F$-closed for any $\fp \in \Spec R$.
\end{proposition}

\begin{proof}
Set $d=\dim R$. Then it suffices to show that if $R$ is parameter $F$-closed, so is $R_{\fp}$ for $\fp \in \Spec R$. Let $I$ be a parameter ideal of $R_{\fp}$. Then we have $I=(a_1,\ldots,a_t)R_{\fp}$, where $t=\Ht \fp$ and $a_i \in R$. As there is nothing to prove when $t=0$, we may assume that $t \ge 1$. Then $\fp$ is minimal over $(a_1,\ldots,a_t)$. Let $J$ be the $\fp$-primary component of $(a_1, \ldots, a_t)$. Then we have $I=JR_{\fp}$, $\Ht J=t$ and $\dim R/J \le d-t$. We also have
$$
(a_1)+J^2 \nsubseteq \bigcup_{\fp \in \Ass R, \dim R/\fp = d} \fp.
$$
By \cite[Theorem 124]{Kap74}, we can find $b_1 \in J^2$ such that
$$
x_1:=a_1+b_1 \notin \bigcup_{\fp \in \Ass R, \dim R/\fp = d} \fp.
$$
Thus, $x_1 \in R$ is a parameter element. For each $1<i \le t$, by the same method, we can find $b_i \in J^2$ such that
$$
x_i:=a_i+b_i \notin \bigcup_{\fp \in \Ass R/(x_1, \ldots, x_{i-1}), \dim R/\fp=d-i+1} \fp.
$$
Therefore, we obtain a part of system of parameters $x_1, \ldots, x_t$ of $R$ such that $x_i = a_i + b_i$ for some $b_i \in J^2$ for all $i=1,\ldots,t$. Then since $b_i \in I^2$ for all $i = 1,\ldots,t$, we have $(x_1,\ldots,x_t) \subseteq I$ and
$$
I=(x_1,\ldots,x_t)+I^2.
$$
So it follows from Nakayama's lemma that $I=(x_1,\ldots,x_t)R_{\fp}$. Extend $x_1,\ldots,x_t$ to a full system of parameters $x_1,\ldots,x_d$. By assumption, the ideal $(x_1,\ldots,x_t, x_t^{n}, \ldots, x_d^n)$ of $R$ is Frobenius closed for all $n \ge 1$. Then by \cite[Lemma 3.1]{M15}, $(x_1,\ldots,x_t)$ is also Frobenius closed. Finally, $I$ is Frobenius closed by Lemma \ref{commute} and we have proved that $R_{\fp}$ is parameter $F$-closed.
\end{proof}

The following proposition allows us to pass to completion when we consider parameter $F$-closed condition.

\begin{proposition}
Let $(R, \fm,k)$ be a local ring of characteristic $p>0$. Then $R$ is parameter $F$-closed if and only if so is the $\fm$-adic completion $\widehat{R}$.
\end{proposition}

\begin{proof}
Let $R \to \widehat{R}$ be the completion map. First, assume that $R$ is parameter $F$-closed. Let $Q$ be a parameter ideal of $\widehat{R}$. Then it is well known that if $Q$ is a parameter ideal of $\widehat{R}$, then there is a parameter ideal $\fq$ of $R$ such that $Q=\fq \widehat{R}$. Let us prove that $\fq \widehat{R}=(\fq\widehat{R})^F$. Since $\fq \widehat{R} \subseteq (\fq\widehat{R})^F$ quite evidently, we prove the other inclusion $(\fq\widehat{R})^F \subseteq \fq \widehat{R}$. This amounts to showing that $\fq\widehat{R}$ is Frobenius closed. To prove this by contradiction, assume that $\fq\widehat{R}$ is not Frobenius closed. Since $(\fq\widehat{R})^F$ is $\fm$-primary, there is an $\fm$-primary ideal $J \subseteq R$ such that $J\widehat{R}=(\fq\widehat{R})^F$. Then the inclusion $\fq \subseteq J$ is proper and thus, there exists an element $x \in ((\fq\widehat{R})^F \setminus \fq\widehat{R}) \cap R$. From this, we infer that $x^q \in \fq^{[q]}\widehat{R} \cap R=\fq^{[q]}$ for $q=p^e \gg 0$. This implies that $x \in \fq$, because $\fq=\fq^F$ by the assumption that $R$ is parameter $F$-closed. This is a contradiction. Hence we get $\fq \widehat{R}=(\fq\widehat{R})^F$. Now we have
$$
Q^F = (\fq\widehat{R})^F=\fq \widehat{R}=Q
$$
and hence, $\widehat{R}$ is parameter $F$-closed.

Assume next that $\widehat{R}$ is parameter $F$-closed. Note that $I^F\subseteq(I\widehat{R})^F \cap R$ for every ideal $I$ of $R$. Moreover by the faithful flatness of $R \to \widehat{R}$, we have $I = I\widehat{R} \cap R$ for every ideal $I$ of $R$. Let $\fq$ be a parameter ideal of $R$. Then we have $\fq^F \subseteq (\fq \widehat{R})^F \cap R=\fq \widehat{R} \cap R = \fq$, which implies that $R$ is parameter $F$-closed.
\end{proof}

The deformation property fails for parameter $F$-closed rings. We are grateful to Linquan Ma for bringing this to our attention and the remark below.

\begin{proposition}
\label{failparameter}
Let $R$ be the ring as in Example \ref{E6.1} with a regular element $y \in \fm$. Then $R/yR$ is parameter $F$-closed, but $R$ is not parameter $F$-closed.
\end{proposition}

\begin{proof}
We have proved that $R/yR$ is $F$-pure and thus, every parameter ideal of $R/yR$ is Frobenius closed, while we have shown that the parameter ideal $(a,x_2^n,x_3^n,x_4^n)$ of $R$ is not Frobenius closed for $n \gg 0$, where the notation is as in Theorem \ref{counterexample}.
\end{proof}

\begin{remark}
We show that the class of parameter $F$-closed local rings and the class of stably $FH$-finite local rings, are not related to each other.
\begin{enumerate}

\item
Let the notation be as in Example \ref{E6.1}. Then $R$ is stably $FH$-finite. However, $R$ is not parameter $F$-closed.

\item
Let us consider the local ring as in \cite[Example 2.16]{EH08}. This ring is Cohen-Macaulay and $F$-injective. Hence it is parameter $F$-closed by Theorem \ref{maintheoremA}. However, its local cohomology modules are not anti-nilpotent.
\end{enumerate}
\end{remark}

\section{A characterization of $F$-injectivity via limit closure}
\label{sec6}

In this section, we give a sufficient and necessary condition of $F$-injectivity via the notion of limit closure which is defined for any Noetherian local ring. For some studies of the limit closure, we refer the reader to \cite{CQ14}, \cite{Hu98} and \cite{MQ16}.

\subsection{Frobenius action on the top local cohomology}
\begin{definition}
\label{limit}
Let $(R,\fm,k)$ be a local ring, let $M$ be a finitely generated module with $d=\dim M$ and let $\underline{x} =x_1,\ldots,x_t$ be a part of system of parameters of $M$. The \textit{limit closure} of $\underline{x}$ in $M$ is defined as a submodule of $M$:
$$
(\underline{x})_M^{\lim}=\bigcup_{n>0}\big((x_1^{n+1},\ldots,x_t^{n+1})M:_M (x_1 \cdots x_t)^n \big)
$$
with the convention that $(\underline{x})_M^{\lim}=0$ when $t=0$. If $M=R$, then we simply write $(\underline{x})^{\lim}$.
\end{definition}

From the definition, it is clear that $(\underline{x})M \subseteq (\underline{x})_M^{\lim}$.

\begin{remark}
\label{limitclosure}
Let the notation be as in Definition \ref{limit}.
\begin{enumerate}
\item
The quotient $(\underline{x})_M^{\lim}/(\underline{x})M$ is the kernel of the canonical map $H^t(\underline{x};M) \cong M/(\underline{x})M \to H^t_{(\underline{x})}(M)$. This implies the following fact. Let $\fq=(x_1,\ldots,x_t)$ and put $\fq_M^{\lim}:=(\underline{x})^{\lim}_M$. Hence the notation $\fq_M^{\lim}$ is independent of the choice of $x_1, \ldots,x_t$ which generate $\fq$.

\item
It is known that $(\underline{x})M=(\underline{x})_M^{\lim}$ if and only if $\underline{x}$ forms an $M$-regular sequence.

\item
It is shown that the Hochster's monomial conjecture is equivalent to the claim that $\fq^{\lim} \neq R$ for every parameter ideal $\fq$ of $R$.
\end{enumerate}
\end{remark}

For a sequence $\underline{x}=x_1,\ldots,x_t$ in a ring $R$, we set $\underline{x}^{[n]}:=x_1^n,\ldots,x_t^n$ and let $(\underline{x}^{[n]})$ be the ideal generated by the sequence $\underline{x}^{[n]}$. We study the Frobenius action on the top local cohomology.

\begin{theorem}
\label{toplc}
Let $(R,\fm,k)$ be a local ring of characteristic $p>0$ and let $\underline{x}=x_1,\ldots,x_t$ be a sequence of elements of $R$ with $(\underline{x}) \subseteq \fm$. Then we have the following statements.

\begin{enumerate}
\item
The Frobenius action on the top local cohomology $H^t_{(\underline{x})}(R)$ is injective if and only if $(\underline{x}^{[n]})^F \subseteq (\underline{x}^{[n]})^{\lim}$ for all $n \ge 1$, where $(\underline{x}^{[n]})^F$ is the Frobenius closure of $(\underline{x}^{[n]})$.

\item
The Frobenius action on the top local cohomology $H^d_{\fm}(R)$ is injective if and only if $\fq^F \subseteq \fq^{\lim}$ for all parameter ideals $\fq$.
\end{enumerate}
\end{theorem}

\begin{proof}
As (2) follows immediately from (1), it is sufficient to prove (1). As in the proof of Theorem \ref{T1.2}, we find that the Frobenius action on $H^t_{(\underline{x})}(R)$ is the direct limit of the following commutative diagram:
$$
\begin{CD}
R/(\underline{x}) @>>>  R/(\underline{x}^{[2]})  @>>> R/(\underline{x}^{[3]}) @>>> \cdots \\
@VFVV @VFVV @VFVV (\star)\\
R/(\underline{x}^{[p]}) @>>> R/(\underline{x}^{[2p]}) @>>> R/(\underline{x}^{[3p]}) @>>> \cdots
\end{CD}
$$
where each vertical map is the Frobenius and each map in the horizontal direction is multiplication by $(x_1 \cdots x_t)$ or $(x_1 \cdots x_t)^p$ in the corresponding spot. The above diagram induces the following commutative diagram
$$
\begin{CD}
R/(\underline{x})^{\lim} @>>>  R/(\underline{x}^{[2]})^{\lim}  @>>> R/(\underline{x}^{[3]})^{\lim} @>>> \cdots \\
@V\overline{F}VV @V\overline{F}VV @V\overline{F}VV  (\star \star)\\
R/(\underline{x}^{[p]})^{\lim} @>>> R/(\underline{x}^{[2p]})^{\lim} @>>> R/(\underline{x}^{[3p]})^{\lim} @>>> \cdots
\end{CD}
$$
where each vertical map $\overline{F}$ is induced by $F$ and every horizontal map is injective.

We first prove the ``only if'' part. Suppose that there is an element $a \in (\underline{x}^{[n]})^F \subseteq (\underline{x})^F$ such that $a \notin (\underline{x}^{[n]})^{\lim}$ for some $n$. We find that the image $\overline{a}$ of $a+ (\underline{x}^{[n]})^{\lim}$ via the map $R/(\underline{x}^{[n]})^{\lim} \to H^t_{(\underline{x})}(R)$ is non-zero by the injectivity of the horizontal maps in $(\star \star)$. On the other hand, $a \in (\underline{x})^F$ implies that $\overline{a}$ is a nilpotent element under the Frobenius action. Then it contradicts the injectivity of Frobenius action on $H^t_{(\underline{x})}(R)$.

We next prove the ``if'' part. Suppose that the Frobenius action $F_*$ on $H^t_{(\underline{x})}(R)$ is not injective. Then there is a non-zero element $\overline{a} \in H^t_{(\underline{x})}(R)$ such that $F_*(\overline{a}) = 0$. Applying the exactness of the direct limit for the diagram $(\star)$, there is an element $a \in R$ together with an integer $n>0$ such that $\overline{a}$ is the canonical image of $a+(\underline{x}^{[n]})$ via the map $R/(\underline{x}^{[n]}) \to H^t_{(\underline{x})}(R)$, and $a+(\underline{x}^{[n]})$ is in the kernel of the map $R/(\underline{x}^{[n]}) \overset{F}{\to} R/(\underline{x}^{[np]})$. Therefore, we have $a^p \in (\underline{x}^{[np]})$ and so $a \in (\underline{x}^{[n]})^F$. However, $\overline{a} \neq 0$ implies that $a \notin (\underline{x}^{[n]})^{\lim}$ and this contradicts the assumption $(\underline{x}^{[n]})^F \subseteq (\underline{x}^{[n]})^{\lim}$.
\end{proof}

\begin{corollary}\label{C7.4}
Let $(R,\fm,k)$ be an $F$-injective local ring of characteristic $p>0$. Then $\fq^F \subseteq \fq^{\lim}$ for all parameter ideals $\fq$.
\end{corollary}

\subsection{$F$-injectivity and Frobenius closure}
We now obtain an ideal-theoretic characterization of $F$-injectivity which is a generalization of Theorem \ref{T1.2}.

\begin{theorem}
\label{characterization}
Let $(R,\fm,k)$ be a local ring of characteristic $p>0$ and of dimension $d>0$. Then the following are equivalent.
\begin{enumerate}
\item
$R$ is $F$-injective

\item
For every filter regular sequence $x_1,\ldots,x_d$, we have
$$
(x_1,\ldots,x_t)^F \subseteq (x_1,\ldots,x_t)^{\lim}
$$
for all $0 \le t \le d$.
\item
There is a filter regular sequence $x_1,\ldots,x_d$ such that
$$
(x_1^n,\ldots,x_t^n)^F \subseteq (x_1^n,\ldots,x_t^n)^{\lim}
$$
for all $0 \le t \le d$ and for all $n \ge 1$.
\end{enumerate}

\end{theorem}

\begin{proof} (2) $\Rightarrow$ (3) is clear since if $x_1, \ldots, x_d$ is a filter regular sequence, then so is $(x_1^n,\ldots,x_d^n)$ for all $n \ge 1$ (cf. Lemma \ref{filter}).

For (3) $\Rightarrow$ (1), the case $t=0$ implies that $(0)^F \subseteq (0) $, so $R$ is reduced. So the Frobenius is injective on $H^0_{\fm}(R)$. Moreover  by Theorem \ref{toplc}, the Frobenius action on $H^d_{\fm}(R)$ is injective. For $0<t<d$, the Frobenius action on $H^t_{(x_1, \ldots,x_t)}(R)$ is injective by Theorem \ref{toplc} again. By Nagel-Schenzel's isomorphism, it follows that $H^t_{\fm}(R) \cong H^0_{\fm}\big(H^t_{(x_1,\ldots,x_t)}(R)\big)$ is an $F$-compatible submodule of $H^t_{(x_1,\ldots,x_t)}(R)$. Hence $R$ is $F$-injective.

We next show that (1) $\Rightarrow$ (2). As in the proof of Theorem \ref{maintheoremA}, using the faithful flatness of
the $\Gamma$-construction (as before, take $\Gamma$ sufficiently small): $R \to \widehat{R}^{\Gamma} \to S:=\widehat{\widehat{R}^{\Gamma}}$, we may assume that $R$ is an $F$-finite $F$-injective local ring. Since $R$ is assumed to be $F$-injective, it is reduced. Therefore, we have the assertion with $t=0$. We also have
$(x_1,\ldots,x_d)^F \subseteq (x_1, \ldots, x_d)^{\lim}$ by Theorem \ref{toplc}. So let us consider the case $0<t < d$. We prove that the Frobenius action on $H^t_{(x_1, \ldots,x_t)}(R)$ is injective. Then the implication (1) $\Rightarrow$ (2) follows from Theorem \ref{toplc}. Suppose that $\overline{a} \in H^t_{(x_1, \ldots,x_t)}(R)$ is nilpotent under the Frobenius action. Let $\fp \in \Spec R$ be such that $(x_1, \ldots, x_t) \subseteq \fp \ne \fm$. Then we find that $R_{\fp}$ is $F$-injective by Lemma \ref{L1.11}. On the other hand, $x_1, \ldots, x_t$ is a regular sequence of $R_{\fp}$ by Lemma \ref{filter} (2). By Proposition \ref{P1.13} and Corollary \ref{Frobeniuslocal} together with their proofs, the Frobenius action on $\big(H^t_{(x_1, \ldots,x_t)}(R)\big)_{\fp} \cong H^t_{(x_1, \ldots,x_t)R_{\fp}}(R_{\fp})$ is injective. Hence we have $\Supp_R(R \cdot \overline{a})=\{\fm \}$. Therefore, $\overline{a} \in H^0_{\fm}\big(H^t_{(x_1, \ldots,x_t)}(R)\big) \cong H^t_{\fm}(R)$. Now the $F$-injectivity of $R$ implies that $\overline{a} = 0$. Thus the Frobenius action on $H^t_{(x_1, \ldots,x_t)}(R)$ is injective. The proof is complete.
\end{proof}

We return to the example considered in the previous section.

\begin{remark}
Le $d=\dim R$. Then the \textit{unmixed component} of $R$, which is denoted by $U_R(0)$, is defined to be the largest submodule of $R$ of dimension less than $d$. If $(0)=\bigcap_{\fp \in \mathrm{Ass}(R)} N(\fp)$ is a reduced primary decomposition of the zero ideal, then $U_R(0)=\bigcap_{\dim R/\fp=d} N(\fp)$. In \cite{CQ14}, Cuong and the first author proved the relation: $U_R(0)=\bigcap_{\fq} \fq^{\lim}$, where $\fq$ runs over all parameters ideals. Now let $R$ be the local ring as in Example \ref{E6.1} and we keep the notation. Then $R$ is not equidimensional, $U_R(0) = (t)$ and $b=u^3z^4t$ is contained in $U_R(0)$. At the time of writing this article, the authors do not have an example of a ring $R$ for which $U_R(0)=(0)$.
\end{remark}

\begin{remark}
It is known that $F$-injective singularities in characteristic $p > 0$ have close connections with Du Bois singularities in characteristic $0$. This connection was studied intensively by Schwede in \cite{Sch09}, where it was proved that in characteristic $0$, those singularities of dense $F$-injective type are Du Bois. It was conjectured that the converse is also true (see \cite{TW14}). More recently, this conjecture has been found to be equivalent to a certain conjecture in arithmetic geometry. This conjecture is considered to reflect deep arithmetic nature of the Frobenius action on sheaf cohomology modules. This was first observed in \cite{MS10} as a weakened version of the \textit{ordinary varieties} due to Bloch and Kato. Then Bhatt, Schwede and Takagi observed its connection with Du Bois and $F$-injective singularities and proposed the \textit{weak ordinarity conjecture} in \cite{BST16}. An interesting observation in \cite{BST16} is that, using Voevodsky's $h$-topology and the sheafification of the structure sheaf on the site (a category with a Grothendieck topology) associated to $h$-topology based on Gabber's idea, Bhatt, Schwede and Takagi have found sheaf theoretic characterizations for both Du Bois and $F$-injective singularities. It will be interesting to know how our ideal theoretic characterization of $F$-injectivity is related to the weak ordinarity conjecture.
\end{remark}

\section{Open problems}
\label{sec7}

We list some open problems in this section.

\begin{Problem}
\label{P1}
\rm Let $I$ be a Frobenius closed parameter ideal of a local ring $(R,\fm,k)$ of characteristic $p>0$. Then is $I^{[q]}$ Frobenius closed for all $q = p^e$?
\end{Problem}

Problem \ref{P1} has an affirmative answer when $R$ is Cohen-Macaulay. It seems to be unknown even when the ring is assumed to be Buchsbaum. This problem is also related to asking that, if there is a Frobenius closed parameter ideal, then is every parameter ideal Frobenius closed?

\begin{Problem}
\label{P2}
Suppose that $(R,\fm,k)$ is an $F$-injective local ring. Then how does one find Frobenius or non-Frobenius closed parameter ideals?
\end{Problem}

\begin{Problem}
\label{P3}
Does there exist a local domain of characteristic $p>0$ that is $F$-injective, but has a parameter ideal that is not Frobenius closed?
\end{Problem}

We state the deformation problems.

\begin{Problem}[Deformation problem I]
\label{P4}
Let $(R,\fm,k)$ be a local ring of characteristic $p>0$ with a regular element $x \in \fm$. Assume that $R/xR$ is $F$-injective (resp. stably $FH$-finite). Then is $R$ also $F$-injective (resp. stably $FH$-finite \footnote{Recently, Ma and the first named-author \cite{MQ17} gave an affirmative answer for the deformation of stably $FH$-finite singularities.})?
\end{Problem}

The first partial result concerning Problem \ref{P4} is obtained in \cite{HMS14}. It is shown in \cite{MSS16} that if $(R,\fm)$ is a local ring essentially of finite type over $\mathbb{C}$ such that $R/xR$ is of dense $F$-injective type for a regular element $x \in \fm$, then $R$ is of dense $F$-injective type.

\begin{Problem}[Deformation problem II]
\label{P5}
Let $(R,\fm,k)$ be an equidimensional local ring of characteristic $p>0$ with a regular element $x \in \fm$. Assume that $R/xR$ is parameter $F$-closed. Then is $R$ also parameter $F$-closed?
\end{Problem}

We note that the example in Proposition \ref{failparameter} is not equidimensional. The authors believe that Problem \ref{P5} has a counterexample.

\begin{Problem}
\label{P6}
What about Problem \ref{P1} through Problem \ref{P5} in the graded local case?
\end{Problem}

\begin{Problem}
\label{P7}
It was shown that the class of parameter $F$-closed rings is strictly contained in the class of $F$-injective local rings. This class contains all $F$-pure local rings. It then becomes a new member in the family of $F$-singularities. In view of the correspondence between the singularities of the minimal model program and the singularities defined by Frobenius map, what is the class of the singularities in the minimal model program corresponding to parameter $F$-closed rings?
\end{Problem}

\end{document}